\documentclass[a4paper, 10pt]{article}
\usepackage{graphicx}

\usepackage{alltt}

\usepackage{amsmath}
\usepackage{epsfig}
\usepackage{graphicx,subfigure,amssymb,amsfonts,latexsym,amsmath,amsthm,enumerate,float}
\usepackage[colorlinks,breaklinks, bookmarks=true]{hyperref}
\hypersetup{colorlinks = true, linkcolor = blue, anchorcolor = blue, citecolor = blue, filecolor = red, urlcolor = red}
\usepackage{hypcap}
\usepackage{parskip}
\usepackage[table]{xcolor}
\definecolor{lightgray}{gray}{0.9}
\usepackage{booktabs,multirow,array}
\usepackage{lineno}
\newtheorem{theorem}{Theorem}

\newtheorem{proposition}[theorem]{Proposition}

\usepackage[raggedright]{titlesec}
\usepackage{lipsum}
\newcommand{\affilmark}[1]{\rlap{\textsuperscript{\itshape#1}}}

\begin{document}

\begin{center}

\LARGE
\textbf{Study of LG-Holling type III predator-prey model with disease in predator}\\[6pt] 
\small
\textbf {Absos Ali Shaikh\footnote{Corresponding author}\affilmark{1}, Harekrishna Das\affilmark{1}, Nijamuddin Ali\affilmark{2}}\\[6pt]
\textsuperscript{1}\ Department of Mathematics, University of Burdwan, \\ Burdwan-713104, West Bengal, India \\ aask2003@yahoo.co.in, hkdasm74@gmail.com\\[6pt]

\textsuperscript{2}\ Department of Mathematics, Katwa College (B.Ed. Section), \\
Burdwan-713130, West Bengal, India.\\
nijamuddin.math@gmail.com\\[6pt]

\end{center}

\begin{abstract}
In this article, a Leslie-Gower Holling type III  predator-prey model with disease in predator has been developed from both biological and mathematical point of view. The total population is divided into three classes, namely, prey, susceptible predator and infected predator. The local stability, global stability together with sufficient conditions for persistence of the ecosystem  near biologically feasible equilibria is thoroughly investigated. Boundedness and existence of the system  are established. All the important analytical findings  are numerically verified using program software MATLAB and Maple. \vskip 2mm


\end{abstract}

\section{Introduction}\label{sec:1}
The predators and the preys carry a dynamic relationship among themselves. And for its universal existence and importance, this relationship is one of the dominant themes in theoretical ecology. Mathematical modelling is considered to be very useful tool to understand and analyze the dynamic behavior of predator-prey systems. Predator functional response on prey population is the major element in predator-prey interaction. It describes the number of prey consumed per predator per unit time for given quantities of prey and predator. The most important and useful functional responses are Lotka-Volterra functional responses such as Holling type I functional response, Holling type II functional response and three species population models with such functional responses are widely researched in ecological literature \cite{pielou1974population}, \cite{murray2001mathematical}, \cite{korobeinikov2005non}, \cite{holling1965functional}. There are also many research works on three species systems like two preys one predator \cite{ali2015stability}, \cite{klebanoff1994chaos}, \cite{gakkhar2003existence}, \cite{el2007chaos}, tritrophic food chain \cite{aziz2002study}, \cite{haque2013study}, \cite{ali2017dynamics} etc.\\
\indent The Mathematical modelling of epidemics has become a very important subject of research after the seminal model of Kermack-MacKendrick (1927) on SIRS (susceptible-infected-removed-susceptible)  systems. It describes the evolution of a disease which gets transmitted upon contact. Important studies have been carried out with the aim of controlling the effects of diseases and of developing suitable vaccination strategies \cite{Anderson1991infectious}, \cite{Li199Global}, \cite{Anderson1982population}. Eco-epidemic research describes  disease that spread among interacting populations, where the epidemic and demographic aspects are merged within one model. During the last decade, this branch of science is developing and studied by the authors in  \cite{Anderson1991infectious}, \cite{Hadeler1989Predator}, \cite{Venturino1994The}. In the natural world, species do not exist alone. It is of more biological significance to study persistence-extinction threshold of each population in systems of two or more interacting species subjected to parasitism. In mathematical biology the predator prey systems and models for transmissible  disease are major field of study in their own right. In the growing ecoepidemic literature and from early papers \cite{Freedman1990Amodel}, disease mainly spreading in the prey are examined in \cite{Xiao2001Modelling}, \cite{Chattapadhyay1999Apredator}, \cite{Chattapadhyay2003Classical}, but in \cite{Venturino2002Epidemics}, \cite{haque2006increase}, \cite{haque2007anecoepi}, the epidemics are  assumed to affect the predators. The predator-prey model with  modified Lesli-Gower Holling type II Scheme was introduced in \cite{Aziz2003Boundedness}, \cite{Guo2008Animpulsive}, \cite{Song2008Dynamic}. The LG model with Holling type II response function with disease in predator is discussed in \cite{Sahabuddin2011Global}. But no one pay the attention for the modified LG model with Holling type III response function for predation with disease in predator.\\
\indent Here we make an attempt to study the above said model with Holling type III response for predation and intra-specific competition among predators. The rest of the article is as follows. In Section $2$, we explain the formulation of the model under consideration and its assumptions. Section $3$ contains some preliminary results. In Section $4$, we  analyze the system behavior of the  trivial equilibria. Also the model with intra specific competition is analyzed for the system behavior around axial and boundary equilibria in Section $4$ . In Section $5$, local and global stability of the  interior equilibria is analyzed. Section $6$ contains persistence of the system. Numerical simulation has been carried out in Section $7$  to support our analytical findings. The article comes to an end with a discussion of the results obtained in Section $8$.

\section{{Mathematical model formulation}}
We make the following assumptions:
\begin{itemize}
\item The disease spreads only among the predators. Let $y$ denotes the susceptible
predators and $z$ the infected ones. The total predator population is $n(t) = y(t) + z(t)$.
\item The disease spreads with a simple mass action law (with the disease incidence $\theta > 0$).
The prey population $x$ grows logistically with intrinsic growth
rate $a_1 > 0$ and carrying capacity $a_1/b_1$ in the absence of predator population. 
\item We introduce intra-specific
competition among the predator's sound and infected sub-populations.
\item Holling type-III response mechanism is considered for predation.
\end{itemize}
%

According to the above assumptions, we get the following model with non negative parameters
\begin{subequations}\label{model-1}
\begin{eqnarray}
&&\frac{dx}{dt}=a_1x-b_1x^2-\frac{c_1x^2y}{k_1+x^2}-\frac{pc_1x^2z}{k_1+x^2}=f_1(x,y,z), \label{model-4-eq1}\\
&&\frac{dy}{dt}=a_2y-\frac{c_2y(y+z)}{k_2+x}-\theta yz =f_2(x,y,z), \label{model-4-eq2}\\
&&\frac{dz}{dt}=\theta yz+a_3z-\frac{c_3z(y+z)}{k_2+x} =f_3(x,y,z), \label{model-4-eq3}\\
&& x(0)\geq 0,\quad y(0)\geq 0,\quad z(0)\geq 0,\nonumber
\end{eqnarray}
\end{subequations}
where $a_2 ,a_3 (a_2\geq a_3) $ are the per capita growth rates of each predator sub population. Thus from sick parents, the disease can be transmitted to their offspring. The parameter $k_1$ represents the half saturation constant of the prey and $k_2$ is the measure of alternative food. Hence the Jacobian matrix of the system (\ref{model-1}) is $J=(m_{ij})\in \mathbb{R}^{3\times 3}$ with entries\\
$m_{11}=a_{1}-2b_{1}x-\frac{2c_{1}xy}{x^2+k1}+\frac{2c_{1}x^3y}{(x^2+k_{1})^2}-\frac{2pc_{1}xz}{x^2+k_{1}}+\frac{2pc_{1}x^3z}{(x^2+k_{1})^2},
\qquad m_{12}=-\frac{c_{1}x^2}{x^2+k_{1}},\qquad m_{13}\\=-\frac{pc_{1}x^2}{x^2+k_{1}},\qquad m_{21}=\frac{c_{2}y(y+z)}{(x+k_{2})^2},\qquad m_{22}=a_{2} -\frac{c_{2}(2y+z)}{x+k_{2}}-\theta z,\qquad m_{23}=-\frac{c_{2}y}{x+k_{2}}-\theta y,\qquad m_{31}=\frac{c_{3}z(y+z)}{(x+k_2)^2},\qquad m_{32}=\theta z-\frac{c_{3}z}{x+k_{2}},\qquad m_{33}=\theta y+a_{3}-\frac{c_{3}(y+2z)}{x+k_{2}}$.
\begin{table}[H]
  \centering
 \caption{The set of model parameters and variables, dimension and their biological description. } 
  \scalebox{0.8}{
  \setlength{\tabcolsep}{20pt}
\begin{tabular}{|c|c|l|}
\hline \parbox[t]{0.7in}{Variable or Parameter} & \parbox[t]{0.7in}{Unit or Dimension} & \parbox[t]{2.5in}{Description} \\ 
\hline $x$ & $V$ & Prey density \\ 
\hline $y$ & $V$ & \parbox[t]{2.5in}{Density of susceptible predator} \\ 
\hline $z$ & $V$ & \parbox[t]{2.5in}{Density of infected predator} \\ 
\hline $a_1$ & $T^{-1}$ & \parbox[t]{2.5in}{Intrinsic growth rate of prey} \\ 
\hline $a_2$ & $T^{-1}$ & \parbox[t]{2.7in}{Intrinsic growth rate of susceptible predator} \\ 
\hline $a_3$ & $T^{-1}$ & \parbox[t]{2.5in}{Intrinsic growth rate of infected predator} \\ 
\hline $b_1$ & $V^{-1}T^{-1}$ & \parbox[t]{2.5in}{Intra-specific competition rate of prey} \\ 
\hline $c_1$ & $T^{-1}$ & \parbox[t]{2.5in}{Predation rate of susceptible predator} \\ 
\hline $c_2$ & $T^{-1}$ & \parbox[t]{2.5in}{Death rate due to intra-specific competition of susceptible predator} \\ 
\hline $c_3$ & $T^{-1}$ & \parbox[t]{2.5in}{Death rate due to intra-specific competition of infected predator} \\ 
\hline $\theta$ & $V^{-1}T^{-1}$ & \parbox[t]{2.5in}{Disease incidence rate }\\ 
\hline $k_1$ & $V^2$ & \parbox[t]{2.5in}{Half saturation constant of the prey} \\ 
\hline $k_2$ & $V$ & \parbox[t]{2.5in}{Measure of alternative food} \\ 
\hline $p$ & $Dimensionless$ & \parbox[t]{2.5in}{Constant lies between 0 to 1} \\
\hline 
\end{tabular}}\\
\end{table}
\section{Preliminary results}\label{PR}
\subsection{Existence}
\begin{theorem}  Every solution of the system (\ref{model-1}) with initial conditions exists  in the interval $(0,+\infty)$ and $x(t)\geq 0$, $y(t)\geq 0$, $z(t)\geq 0$ for all $t\geq 0$.
\end{theorem}

\begin{proof} 
We have $\frac{dx}{dt}=f_1(x,y,z)$,  $\frac{dy}{dt}=f_2(x,y,z)$, $\frac{dz}{dt}=f_3(x,y,z)$. Integrating we get 
$x(t)=x(0)e^{\int_{0}^{t} f_1(x,y,z) ds}$, $y(t)=y(0)e^{\int_{0}^{t} f_2(x,y,z) ds}$, $z(t)=z(0)e^{\int_{0}^{t} f_3(x,y,z) ds}$, 
where $x(0)=x_0>0$, $y(0)=y_0>0$, $z(0)=z_0>0$. Since $f_1, \ f_2, \ f_3$ are continuous function and hence locally Lipschitzian on $R^3_+$, the solution with positive initial condition  exists and unique on $(0,\xi)$ where $0<\xi<\infty$.
 Hence the theorem.
\end{proof}

\subsection{Boundedness}

\begin{theorem} 
All the solutions of the system which initiate in $\mathbb{R}_{+}^{3}$ are uniformly bounded.
\end{theorem}

\begin{proof} Let us define a function
$ \omega=x+y+z$. Therefore, we have\\
\begin{eqnarray*}
\frac{d\omega}{dt}+\mu\omega
&=&\frac{dx}{dt}+\frac{dy}{dt}+\frac{dz}{dt}+\mu(x+y+z)\\
 &=&[a_1x(1-\frac{b_1}{a_1}x)+\mu x]-\frac{(c_1y+pc_1z)x^2}{x^2+k_1}-\frac{(c_2y+c_3z)(y+z)}{x+k_2}\\
 &+&(a_2+\mu)y+(a_3+\mu)z\\
& \leq & (a_1+\mu-b_1x)x+(a_2+\mu)y+(a_3+\mu)z\\
&\leq&(\frac{(a_1+\mu)^2}{4b_1})+(a_2+\mu)y+(a_3+\mu)z 
\text{ for each } \mu>0.
\end{eqnarray*}
Hence we find $l>0$ such that $\frac{dw}{dt}+\mu\omega \leq l $, $\forall t \in (0,t_b)$.
Using the theory of differential inequality \cite{birkhoffordinary}, we obtain $0<\omega(x,y,z)\leq\frac{l}{\mu}(1-e^{-ut})+\omega\Big(x(0),y(0),z(0)\Big)e^{-\mu t}$ and for $t_b\rightarrow\infty$  we have $0<\omega\leq \frac{l}{\mu}$.\\
Hence all the solutions of the system that initiate in $\mathbb{R}_{+}^{3}$ are confined in the region $\gamma=\{{(x,y,z)\in\mathbb{R}_{+}^{3}:\omega=\frac{l}{\mu}+\epsilon}\}$ for any $\epsilon>0$ and for $t$ large enough.\\
Hence the theorem.
\end{proof}

\subsection{Equilibrium points}
The system of equations (\ref{model-1}) has the equilibrium points $E_0(0,0,0)$, $E_1(0,0,\frac{a_{3}k_{2}}{c_{3}})$, $E_2(0,\frac{a_2 k_2}{c_2},0)$,
$E_3(0,y_{3},z_3)$, $E_4(\frac{a_{1}}{b_{1}},0,0)$, 
$E_5(x_5,y_5,0)$, $E_6(x_6,0,z_6)$ and $E_*(x_*,y_*,z_*)$. The co-existence equilibrium is $E_*(x_*,y_*,z_*)$ where 
$y_*=\frac{-a_3\theta x_*-a_3k_2\theta+a_2c_3-a_3c_2}{\theta(\theta x_*+k_2\theta+c_2-c_3)}$ ,\\ $z_*=-\frac{-a_2\theta x_*-a_2k_2\theta+a_2c_3-a_3c_2}{\theta(\theta x_*+k_2\theta+c_2-c_3)}$ and  $x_*$ is root of the following equation\\
$P(x)=A_4x^4+A_3x^3+A_2x^2+A_1x+A_0=0$.\\
 Here $A_4=b_1\theta^2$,\\
 $A_3=b_1k_2\theta^2-a_1\theta^2+b_1c_2\theta-b_1c_3\theta$,\\
 $A_2=-a_1k_2\theta^2+a_2c_1\theta p+b_1k_1\theta^2-a_1c_2\theta+a_1c_3\theta-a_3c_1\theta$,\\
 $A_1=a_2c_1k_2\theta p+b_1k_1k_2\theta^2-a_1k_1\theta^2-a_2c_1c_3p+a_3c_1c_2p-a_3c_1k_2\theta+b_1c_2k_1\theta-b_1c_3k_1\theta+a_2c_1c_3-a_3c_1c_2$,\\
 $A_0=-a_1k_1k_2\theta^2-a_1c_2k_1\theta+a_1c_3k_1\theta$.\\ We consider $P(x)=(\gamma_1x^2+\delta_1x+\phi_1)(\gamma_2x^2+\delta_2x+\phi_2)$.\\
{\bf Case-1}: $P(x)$ has two real roots if either
$\delta_1^2-4\gamma_1\phi_1>0$ or $\delta_2^2-4\gamma_2\phi_2>0$.\\ For the set of parameters $ a_1 = 4.5, \ a_2 = 3.8, \   a_3 = 0.005, \  b_1 = 0.075, \ k_1 = 100, \ k_2= 160, \ c_1 = 2.8, \ c_2 = 1.97, \ c_3 = 1.95, \ \theta= 0.0937, \ p= 0.047$, there are two real roots $[x = 56.43479200, \ y = 3.837197569, \ z = 36.62450011]$ and $[x = -161.0505378, \ y = -1006.972848, \ z = 1057.801828] $ in which first one is biologically feasible.\\
{\bf Case-2}: $P(x)$ has four real roots if
$\delta_1^2-4\gamma_1\phi_1>0$ and $\delta_2^2-4\gamma_2\phi_2>0$.\\
For  the set of parameters $a_1 = 5.0, \ a_2 = 7.8, \ a_3 = 1.5, \ b_1 = 0.0005, \ k_1 = 50, \ k_2 = 55, \ c_1 = 1.7, \ c_2 = 1.05, \ c_3 = 1.0, \ \theta = 0.0217, \ p = 0.73 $, there are four real roots $[x = 0.5990751338, \ y = 161.9321592, \ z = 116.8375923]$, $[x = -65.33582457, \ y = -1734.893676, \ z = 2108.504719]$, $[x = 73.69033011, \ y = 33.00904773, \ z = 252.2068593]$, $[x = 9933.742272, \ y = -67.78533292, \ z = 358.0409590]$ in which first and third one  are biologically feasible.\newline

\par For the equilibrium point $E_5(x_5,y_5,0)$, $z_5=0$ gives $y_5=\frac{a_2(x_5+k_2)}{c_2}$. Here $x_5$ is the root of \\$Q(x)=A_3x^3+A_2x^2+A_1x+A_0=0$, \\where $A_3=b_1c_2$, 
$A_2=-a_1c_2+a_2c_1$, 
$A_1=a_2c_1k_2+b_1c_2k_1$,
$A_0=-a_1c_2k_1$.\\
We consider $Q(x)=(\mu_1 x+\alpha_1)(\nu_1 x^2+\xi_1 x+\eta_1)$.\\
{\bf Case-1}: $Q(x)$ has only one real root if 
$\xi_1^2-4\nu_1\eta_1<0$, which yields $x_5=-\frac{\alpha_1}{\mu_1}$. \\For the set of parameters $ a_1 = 4.5, \ a_2 = 3.8, \ a_3 = 0.005, \ b_1 = 0.075, \ k_1 = 100, \ k_2 = 160, \ c_1 = 2.8, \ c_2 = 1.97, \ c_3 = 1.95, \ \theta = 0.0937, \ p = 0.047$, there is only one real root $[x = 0.5159678886, \ y = 309.6247096, \ z = 0]$ which is biologically feasible. \\
{\bf Case-2}: $Q(x)$ has three real roots if
$\xi_1^2-4\nu_1\eta_1>0$.\\ For the set of parameters $a_1 = 5.0, \ a_2 = 7.8, \ a_3 = 1.5, \ b_1 = 0.0005, \ k_1 = 50, \ k_2 = 55, \ c_1 = 1.7, \ c_2 = 1.05, \ c_3 = 1.0, \ \theta = 0.0217, \ p = 0.73 $, there are three real roots $[x = 0.3585095936, \ y = 411.2346427, \ z = 0]$, $[x = -91.96262935, \ y = -274.5795323, \ z = 0]$, $[x=-15165.53874, \ y = -1.122497163*10^5, \ z = 0]$ in which first one is biologically feasible.\newline

\par For the equilibrium point $E_6(x_6,0,z_6)$, $y_6=0$ gives $z_6=\frac{a_3(x_6+k_2)}{c_3}$. Here $x_6$ is the root of \\$R(x)=A_3x^3+A_2x^2+A_1x+A_0=0$,\\ where 
$A_3=b_1c_3$, 
$A_2=-a_1c_3+a_3c_1p$,
$A_1=a_3c_1k_2p+b_1c_3k_1$,
$A_0=-a_1c_3k_1$.\\
We consider $R(x)=(\mu_2 x+\alpha_2)(\nu_2 x^2+\xi_2 x+\eta_2)$.\\
{\bf Case-1}: $R(x)$ has only one real root if 
$\xi_2^2-4\nu_2\eta_2<0$, which yields $x_6=-\frac{\alpha_2}{\mu_2}$.\\ For the set of parameters $ a_1 = 4.5, \ a_2 = 3.8, \ a_3 = 0.005, \ b_1 = 0.075, \ k_1 = 100, \ k_2 = 160, \ c_1 = 2.8, \ c_2 = 1.97, \ c_3 = 1.95, \ \theta = 0.0937, \ p = 0.047$, there is only one real root  $[x = 59.98394610, \ y = 0, \ z = 0.5640614003]$ which is biologically feasible.\\
{\bf Case-2}: $R(x)$ has three real roots if
$\xi_2^2-4\nu_2\eta_2>0$.\\ For the set of parameters $a_1 = 5.0, \ a_2 = 7.8, \ a_3 = 1.5, \ b_1 = 0.0005, \ k_1 = 50, \ k_2 = 55, \ c_1 = 1.7, \ c_2 = 1.05, \ c_3 = 1.0, \ \theta = 0.0217, \ p = 0.73 $, there are three real roots $[x = 2.657590193, y = 0, z = 86.48638529]$, $[x = 30.13036196, \ y = 0, \ z = 127.6955429]$, $[x = 6244.212048, \ y = 0, \ z = 9448.818072]$ which are biologically feasible.


\section{System behaviour around boundary equilibria}
\subsection{Stability for $E_0$}
The characteristic equation for $E_0$ is given by  \\ \newline
$\left|\begin{array}{ccc} 
a_1-\lambda & 0 & 0 \\ 
0 & a_2-\lambda & 0\\
0 & 0 & a_3-\lambda\\
\end{array}
\right|=0.$\\ \newline
The equilibrium point $E_0$ has the eigenvalues $a_1$, $a_2$, $a_3$. All the eigenvalues are positive and it is unstable.

\subsection{Stability for $E_1$}
The characteristic equation for $E_1$ is given by\\ \newline
$\left|\begin{array}{ccc} 
a_1-\lambda & 0 & 0 \\ 
0 & (\frac{a_2c_3-a_3(c_2+k_2\theta)}{c_3})-\lambda & 0\\
\frac{a_3^2}{c_3} & a_3(-1+\frac{k_2\theta}{c_3}) & -a_3-\lambda\\
\end{array}
\right|=0.$\\ \newline
Since one of the eigenvalues of $E_1$ is $a_1$, which is always positive and so, $E_1$ is unstable.

\subsection{Stability for $E_2$}
The characteristic equation for $E_2$ is given by \\ \newline
$\left|\begin{array}{ccc} 
a_1-\lambda & 0 & 0 \\ 
\frac{a_2^2}{c_2} & -a_2-\lambda & -a_2
-\frac{a_2k_2\theta}{c_2}\\
0 & 0 & (a_3-\frac{a_2c_3}{c_2}+\frac{a_2k_2\theta}{c_2})-\lambda\\
\end{array}
\right|=0.$\\ \newline
Since one of the eigenvalues of $E_2$ is $a_1$, which is always positive and therefore, $E_2$ is unstable.

\subsection{Stability for $E_3$}
The characteristic equation of the equilibrium $E_3(0,y_3,z_3)$ is given by\\ \newline
$\left|\begin{array}{ccc} 
a_1-\lambda & 0 & 0 \\ 
\frac{c_2y_3(y_3+z_3)}{k_2^2} & (a_2-\frac{c_2(2y_3+z_3)}{k_2}-z_3 \theta)-\lambda & -\frac{c_2y_3}{k_2}-y_3 \theta\\
\frac{c_3z_3(y_3+z_3)}{k_2^2} & -\frac{c_3z_3}{k_2}+z_3 \theta & (a_3-\frac{c_3(y_3+2z_3)}{k_2}+y_3 \theta)-\lambda\\
\end{array}
\right|=0,$ \newline
where 
$y_3=\frac{-a_3k_2\theta+a_2c_3-a_3c_2}{\theta(k_2\theta+c_2-c_3)}$ and $z_3=-\frac{-a_2k_2\theta+a_2c_3-a_3c_2}{\theta(k_2\theta+c_2-c_3)}$. As eigenvalue $a_1$ for $E_3$ is always positive, the equilibrium point $E_3$ is unstable.

\subsection{Stability for $E_4$}
The characteristic equation for $E_4$ is given by\\ \newline
$\left|\begin{array}{ccc} 
-a_1-\lambda & -\frac{a_1^2c_1}{b_1^2(\frac{a_1^2}{b_1^2}+k_1)} & -\frac{a_1^2c_1p}{b_1^2(\frac{a_1^2}{b_1^2}+k_1)} \\ 
0 & a_2-\lambda & 0\\
0 & 0 & a_3-\lambda\\
\end{array}
\right|=0.$\\ \newline
Eigenvalues $a_2$ and $a_3$ are always positive of the equilibrium $E_4$. Hence $E_4$ is unstable.

\subsection{Stability for $E_5(x_5,y_5,0)$}\label{E_5}
At $E_5(x_5,y_5,0)$, the Jacobian matrix for the system is given by \newline
\[J_5=\left(\begin{array}{ccc}
m_{11} & m_{12} & m_{13}\\
m_{21} & m _{22} & m_{23}\\
m_{31} & m_{32} & m_{33}\\
\end{array}
\right),\] \\ \newline
where $m_{11}=-b_1x_5-\frac{c_1x_5y_5}{x_5^2+k_1}+\frac{2c_1x_5^3y_5}{(x_5^2+k_1)^2}$,\qquad$m_{12}=-\frac{c_1x_5^2}{x_5^2+k_1}$,\qquad$m_{13}=-\frac{pc_1x_5^2}{x_5^2+k_1}$,\qquad$m_{21}=\frac{c_2y_5^2}{(x_5+k_2)^2}$,\qquad$m_{22}=-\frac{c_2y_5}{x_5+k_2}$,\qquad$m_{23}=-\frac{c_2y_5}{x_5+k_2}-\theta y_5$,\qquad $m_{31}=0$,\qquad$m_{32}=0$,\qquad $m_{33}=\theta y_5+a_3-\frac{c_3y_5}{x_5+k_2}$.\\ \newline
The Characteristic equation for $J_5$ is given by \\ 
$(m_{33}-\lambda)\{(m_{11}-\lambda)(m_{22}-\lambda)-m_{21}m_{12}\}= 0$\\
$\Rightarrow (m_{33}-\lambda)\{\lambda^2-(m_{11}+m_{22})\lambda+m_{11}m_{22}-m_{21}m_{12}\}=0$\\
$\Rightarrow\lambda_{1,2}=\frac{m_{11}+m_{22}\pm\sqrt{(m_{11}+m_{22})^2-4(m_{11}m_{22}-m_{21}m_{12})}}{2}$ and $\lambda_3=m_{33}$.\\
We choose $m_{11}<0$ and $m_{33}<0$. Then $E_5$ will be stable if\\
(i) $b_1x_5+\frac{c_1x_5y_5}{x_5^2+k_1}>\frac{2c_1x_5^3y_5}{(x_5^2+k_1)^2}$,\\
(ii)$\frac{c_3y_5}{x_5+k_2}>\theta y_5+a_3$.
\subsection{Stability for $E_6(x_6,0,z_6)$}\label{E_6}
At $E_6(x_6,0,z_6)$, the Jacobian matrix for the system is given by\\
\[J_6=\left(\begin{array}{ccc}
m_{11} & m_{12} & m_{13}\\
m_{21} & m _{22} & m_{23}\\
m_{31} & m_{32} & m_{33}\\
\end{array}
\right),\] \\ where $m_{11}=-b_1x_6-\frac{pc_1x_6z_6}{x_6^2+k_1}+\frac{2pc_1x_6^3z_6}{(x_6^2+k_1)^2}$,\qquad$m_{12}=-\frac{c_1x_6^2}{x_6^2+k_1}$,\qquad$m_{13}=-\frac{pc_1x_6^2}{x_6^2+k_1}$,\qquad$m_{21}=0$,\qquad$m_{22}=a_2-\frac{c_2z_6}{x_6+k_2}-\theta z_6$,\qquad$m_{23}=0$,\qquad $m_{31}=\frac{c_3z
_6^2}{(x_6+k_2)^2}$,\qquad$m_{32}=\theta z_6-\frac{c_3z_6}{x_6+k_2}$,\qquad $m_{33}=-\frac{c_3z_6}{x_6+k_2}$.\\ \newline
The Characteristic equation for $J_6$ is given by \\
$(m_{22}-\lambda)\{(m_{11}-\lambda)(m_{33}-\lambda)-m_{31}m_{13}=0$\\
$\Rightarrow (m_{22}-\lambda)\{{\lambda^2-(m_{11}+m_{33})\lambda+m_{11}m_{33}-m_{31}m_{13}}\}=0$\\
$\Rightarrow\lambda_{1,2}=\frac{m_{11}+m_{33}\pm\sqrt{(m_{11}+m_{33})^2-4(m_{11}m_{33}-m_{31}m_{13})}}{2}$ and $\lambda_3=m_{22}$.\\
We choose $m_{22}<0$ and $m_{11}<0$. Then $E_6$ will be stable if\\
(i) $b_1x_6+\frac{pc_1x_6z_6}{x_6^2+k_1}>\frac{2pc_1x_6^3z_6}{(x_6^2+k_1)^2}$,\\
(ii)$\frac{c_2z_6}{x_6+k_2}+\theta z_6>a_2$.

\section{System behaviour near the coexistence equilibrium $E_*(x_*,y_*,z_*)$}\label{E_*}
The entries for the Jacobian at $E_*(x_*,y_*,z_*)$ are\\
$m_{11}=a_1-2b_1x_*+\frac{2c_1x_*^3y_*}{(x_*^2+k_1)^2}+\frac{2pc_1x_*^3z_*}{(x_*^2+k_1)^2}$, $m_{12}=-\frac{c_1x_*^2}{x_*^2+k_1}$, $m_{13}=-\frac{pc_1x_*^2}{x_*^2+k_1}$,\\ $m_{21}=\frac{c_2y_*(y_*+z_*)}{(x_*+k_2)^2}$, $m_{22}=-\frac{c_2y_*}{x_*+k_2}$, $m_{23}=-\frac{c_2y_*}{x_*+k_2}-\theta y_*$,\\ $m_{31}=\frac{c_3z
_*(y_*+z_*)}{(x_*+k_2)^2}$, $m_{32}=\theta z_*-\frac{c_3z_*}{x_*+k_2}$, $m_{33}=-\frac{c_3z_*}{x_*+k_2}$.

\subsection{Local Stability}
The characteristic equation for $J_*$ is 
$\theta^3+A_1\theta^2+A_2\theta+A_3=0$, where\\ $A_1=-m_{11}-m_{22}-m_{33}$, \\
$A_2=m_{11}m_{22}+m_{11}m_{33}+m_{22}m_{33}-m_{13}m_{31}-m_{23}m_{32}-m_{21}m_{12}$, \\
$A_3=m_{11}m_{23}m_{32}+m_{12}m_{21}m_{33}+m_{13}m_{22}m_{31}-m_{11}m_{22}m_{33}-m_{12}m_{23}m_{31}-m_{13}m_{21}m_{32}$,\\
$A_1A_2-A_3=-m_{11}^2m_{22}-m_{11}^2m_{33}-m_{22}^2m_{33}-m_{11}m_{22}^2-m_{11}m_{33}^2-m_{22}m_{33}^2-2m_{11}m_{22}m_{33}+m_{11}m_{13}m_{31}+m_{11}m_{12}m_{21}+m_{22}m_{12}m_{21}+m_{22}m_{23}m_{32}\\+m_{33}m_{23}m_{32}+m_{33}m_{13}m_{31}+m_{12}m_{23}m_{31}+m_{13}m_{21}m_{32}$. \\
Assume $(i)$ $m_{11}<0$, 
$(ii)$ $m_{32}>0$,
$(iii)$ $m_{11}m_{31}+m_{33}m_{31}+m_{21}m_{32}<0$ or $m_{11} m_{12}+m_{22}m_{12}+m_{13}m_{32}>0$ or $m_{22}m_{23}+m_{33}m_{23}+m_{13}m_{21}>0$ and we have\\ $A_1>0$,\ $A_3>0$ and $A_1A_2-A_3>0$.
By Routh-Hurwitz criterion, the interior or co-existence equilibrium $E_*(x_*,y_*,z_*)$ is  locally asymptotically stable.

\subsection{Global Stability}
\begin{theorem} The co-existence equilibrium point $E_*$ is globally asymptotically stable if $P_1>0$, $P_2>0$ and $P_3>0$ where $P_1$ , $P_2$, $P_3$ are defined latter.
\end{theorem}
\begin{proof} Let us define the function $L(x,y,z)=L_1(x,y,z)+L_2(x,y,z)+L_3(x,y,z)$, \\
where
$L_1=x-x_*-x_*ln\frac{x}{x_*}$, \ 
$L_2=y-y_*-y_*ln\frac{y}{y_*}$, \ $L_3=z-z_*-z_*ln\frac {z}{z_*}$.\\
It is to be shown that $L$ is a Lyapunav function and $L$ vanishes at $E_*$ and it is positive for all $x,y,z>0$.  Hence $E_*$ represents its global minimum. We have
\begin{eqnarray*}
\frac{dL_1}{dt}&=&(x-x_*)\Big(a_1-b_1x-\frac{c_1xy}{x^2+k_1}-\frac{pc_1xz}{x^2+k_1}\Big)\\
&=&(x-x_*)\Big(b_1x_*+\frac{c_1x_*y_*}{x_*^2+k_1}+\frac{pc_1x_*z_*}{x_*^2+k_1}-b_1x-\frac{c_1xy}{x^2+k_1}-\frac{pc_1xz}{x^2+k_1}\Big)\\
&=&(x-x_*)\Big[\frac{c_1(x-x_*)(xx_*-k_1)(y_*+p z_*)}{(x_*^2+k_1)(x^2+k_1)}-b_1(x-x_*)-\frac{c_1x(y-y_*)}{x^2+k_1}\\
&-&\frac{pc_1x(z-z_*)}{x^2+k_1}\Big],
\end{eqnarray*}

\begin{eqnarray*}
\frac{dL_2}{dt}&=&(y-y_*)\Big(a_2-\frac{c_2(y+z)}{x+k_2}-\theta z\Big)\\
&=&(y-y_*)\Big(\theta z_*+\frac{c_2(y_*+z_*)}{x_*+k_2}-\frac{c_2(y+z)}{x+k_2}-\theta z\Big)\\
&=&(y-y_*)\Big[-\theta(z-z_*)+\frac{c_2(y_*+z_*)(x-x_*)}{(x_*+k_2)(x+k_2)}-\frac{c_2((y-y_*)+(z-z_*))}{x+k_2}\Big],
\end{eqnarray*}

\begin{eqnarray*}
\frac{dL_3}{dt}&=&(z-z_*)\Big(a_3-\frac{c_3(y+z)}{x+k_2}+\theta y\Big)\\
&=&(z-z_*)\Big(-\theta z_*+\frac{c_3(y_*+z_*)}{x_*+k_2}-\frac{c_3(y+z)}{x+k_2}+\theta y\Big)\\
&=&(z-z_*)\Big[\theta(y-y_*)+\frac{c_3(y_*+z_*)(x-x_*)}{(x_*+k_2)(x+k_2)}-\frac{c_3((y-y_*)+(z-z_*))}{x+k_2}\Big].
\end{eqnarray*}
We consider
\begin{eqnarray*}
\frac{dL}{dt}&=&-\Big[A(x-x_*)^2+B(y-y_*)^2+C(z-z_*)^2+2H(x-x_*)(y-y_*)\\
&+& 2F(y-y_*)(z-z_*)+2G(z-z_*)(x-x_*)\Big]=-V^T Q V
\end{eqnarray*}

where $V=\Big((x-x_*),(y-y_*),(z-z_*)\Big)^T$ and $Q$ is symmetric quadratic form given by\\
$Q=
\left(\begin{array}{ccc}
A & H & G\\
H & B & F\\
G & F & C
\end{array}\right)$\\
with the entries that are functions only of the variable $x$ and \\
$A=b_1-\frac{c_1(y_*+pz_*)(xx_*-k_1)}{(x_*^2+k_1)(x^2+k_1)}$, $ B=\frac{c_2}{x+k_2}$, $C=\frac{c_3}{x+k_2}$, $F=\frac{c_2+c_3}{2(x+k_2)}$,\\
$H=\frac{1}{2}[\frac{c_1x}{x^2+k_1}-\frac{c_2(y_*+z_*)}{(x_*+k_2)(x+k_2)}]$, $G=\frac{1}{2}[\frac{pc_1x}{x^2+k_1}-\frac{c_3(y_*+z_*)}{(x_*+k_2)(x+k_2)}]$.

If the matrix $Q$ is positive definite, then
$\frac{dL}{dt}<0$. So, all the principal minors of $Q$, namely, $P_1\equiv A$, \ $P_2 \equiv AB-H^2 $, \ $P_3\equiv ABC+2FGH-AF^2-BG^2-CH^2=C(AB-H^2)+G(FH-BG)+F(GH-AF)$ to be positive
, i.e. , $P_1>0, \ P_2>0, \ P_3>0$.

\end{proof}

\section{Persistence}
If a compact set $D\subset\Omega=\{(x,y,z): x>0, y>0, z>0\}$ exists such that all solution of the system (\ref{model-1}) eventually enter and remain in $D$, the system is called persistent.
\\

\begin{proposition} The system (\ref{model-1}) is persistent if\\
\begin{enumerate}
\item $\frac{a_2c_3}{a_3}>(c_2+k_2\theta)$,
\item $a_3c_2>a_2c_3$,
\item $a_3c_2>a_2c_3, \ a_1>b_1x_5+\frac{a_2c_1x_5(k_2+x_5)}{c_2(k_1+x_5^2)}$,
\item $\frac{a_2c_3}{a_3}>\Big(c_2+(k_2+x_6)\theta \Big), \ a_1c_3(k_1+x_6^2)>x_6\Big(a_3c_1p(k_2+x_6)+b_1c_3(k_1+x_6^2)\Big)$.
\end{enumerate}
\end{proposition}
\begin{proof}  Let us consider the method of  average Lyapunav function, see \cite{Gard1979persistence} , considering a function of the form $V(x,y,z)=x^{\gamma_1}y^{\gamma_2}z^{\gamma_3}$, where $\gamma_i=1,2,3$ are positive constant to be determined. We define \\
\begin{equation*}
\begin{aligned}
\Pi(x,y,z)&=\frac{\dot{V}}{V}\\
&=\gamma_1\Big(a_1-b_1x-\frac{c_1xy}{x^2+k_1}-\frac{pc_1xz}{x^2+k_1}\Big)
&+\gamma_2\Big(a_2-\frac{c_2(y+z)}{x+k_2}-\theta z\Big)\\
&+\gamma_3\Big(a_3-\frac{c_3(y+z)}{x+k_2}+\theta y\Big).\\
\end{aligned}
\end{equation*}
We are to prove that this function is positive at each boundary equilibrium. We have
$\Pi(0,0,0)=\gamma_1a_1+\gamma_2a_2+\gamma_3a_3>0$ and $\Pi(\frac{a_1}{b_1},0,0)=\gamma_2a_2+\gamma_3a_3>0$.\\Here 
$\Pi(0,0,\frac{a_3k_2}{c_3})=a_1\gamma_1+\frac{a_2c_3-a_3(c_2+k_2\theta)}{c_3}\gamma_2>0$ follows by condition 1.\\
With the condition 2, we have \\
$\Pi(0,\frac{a_2k_2}{c_2},0)=a_1\gamma_1+\frac{(a_3c_2-a_2c_3+a_2k_2\theta)}{c_2}\gamma_3>0$. Also \\
\begin{equation*}
\begin{aligned}
&\Pi\Big(0,\frac{a_2c_3-a_3c_2-a_3k_2\theta}{\theta(k_2\theta+c_2-c_3)},-\frac{a_2c_3-a_3c_2-a_2k_2\theta}{\theta(k_2\theta+c_2-c_3)}\Big)\\
&=a_1 \gamma_1>0.
\end{aligned}
\end{equation*}
 We find
 \begin{equation*}
\begin{aligned}
\Pi\Big(x_5,\frac{a_2(x_5+k_2)}{c_2},0\Big)&=\gamma_1\Big(a_1-\Big(b_1x_5+\frac{a_2c_1x_5(k_2+x_5)}{c_2(k_1+x_5^2)}\Big)\Big)\\
&+\gamma_3\Big(\frac{a_3c_2-a_2c_3+a_2k_2\theta+a_2x_5\theta}{c_2}\Big)>0 \ \text{by the condition } 3,
\end{aligned}
\end{equation*}

\begin{equation*}
\begin{aligned}
\Pi\Big(x_6,0,\frac{a_3(x_6+k_2)}{c_3}\Big)&=\gamma_1\Big(\frac{a_1c_3(k_1+x_6^2)-x_6(a_3c_1p(k_2+x_6)+b_1c_3(k_1+x_6^2))}{c_3(k_1+x_6^2)}\Big)\\
&+\gamma_2\Big(\frac{a_2c_3-a_3(c_2+(k_2+x_6)\theta)}{c_3}\Big)>0 \ \text{ by the condition } 4.
\end{aligned}
\end{equation*}
Hence a suitable choice of $\gamma_1,\gamma_2,\gamma_3$ is required to ensure $\Pi>0$ at the boundary equilibria.
Hence $V$ is an average Lyapunav function and thus the system (\ref{model-1}) is persistent.
\end{proof}


\section{Numerical simulation}
Analytical studies can never be completed without numerical verification of the derived results. In this section, we present computer simulations of some solutions of the system (\ref{model-1}). Beside verification of our analytical findings, these numerical simulations are very important from practical point of view. We use four different set of numerical values for support of analytical results mentioned in Table \ref{LG_typeIII_set of parameters}.

\begin{table}[H]
\caption{Set of parameter values for numerical simulations; $S\equiv $Parameter sets}\label{LG_typeIII_set of parameters}
\setlength{\tabcolsep}{5pt}
\begin{tabular}{ | l l l l l l l l l l l l|}
\hline
\parbox[t]{0.021in}{$S$} &\parbox[t]{0.021in}{$a_1$} & \parbox[t]{0.021in}{$a_2$} & \parbox[t]{0.021in}{$a_3$} &\parbox[t]{0.021in}{ $b_1$} & \parbox[t]{0.021in}{$k_1$} & \parbox[t]{0.021in}{$k_2$} & \parbox[t]{0.021in}{$c_1$} & \parbox[t]{0.021in}{$c_2$} & \parbox[t]{0.021in}{$c_3$} & \parbox[t]{0.021in}{$\theta$} & \parbox[t]{0.021in}{$p$ }\\ \hline
$S_1$ & $4.5$ & $3.8$ & $0.005$ & $0.075$ & $100$ & $160$ & $2.8$ & $1.97$ & $1.95$ & $0.0937$ & $0.047$ \\ 
$S_2$ & $4.5$ & $3.8$ & $0.005$ & $0.075$ & $100$ & $20$ & $2.8$ & $1.97$ & $0.005$ & $0.0937$ & $0.047$ \\ 
$S_3$ & $5.0$ & $7.8$ & $1.5$ & $0.0005$ & $50$ & $55$ & $1.7$ & $1.05$ & $1.0$ & $0.0217$ & $0.73$ \\ 
$S_4$ & $4.0$ & $6.0$ & $0.05$ & $0.005$ & $100$ & $200$ & $0.08$ & $0.7$ & $0.50$ & $0.002537$ & $0.93$ \\ \hline

\end{tabular}
\end{table}

\begin{figure}[H]
\centering
(a)\includegraphics[width=5.5cm,height=3.5cm]{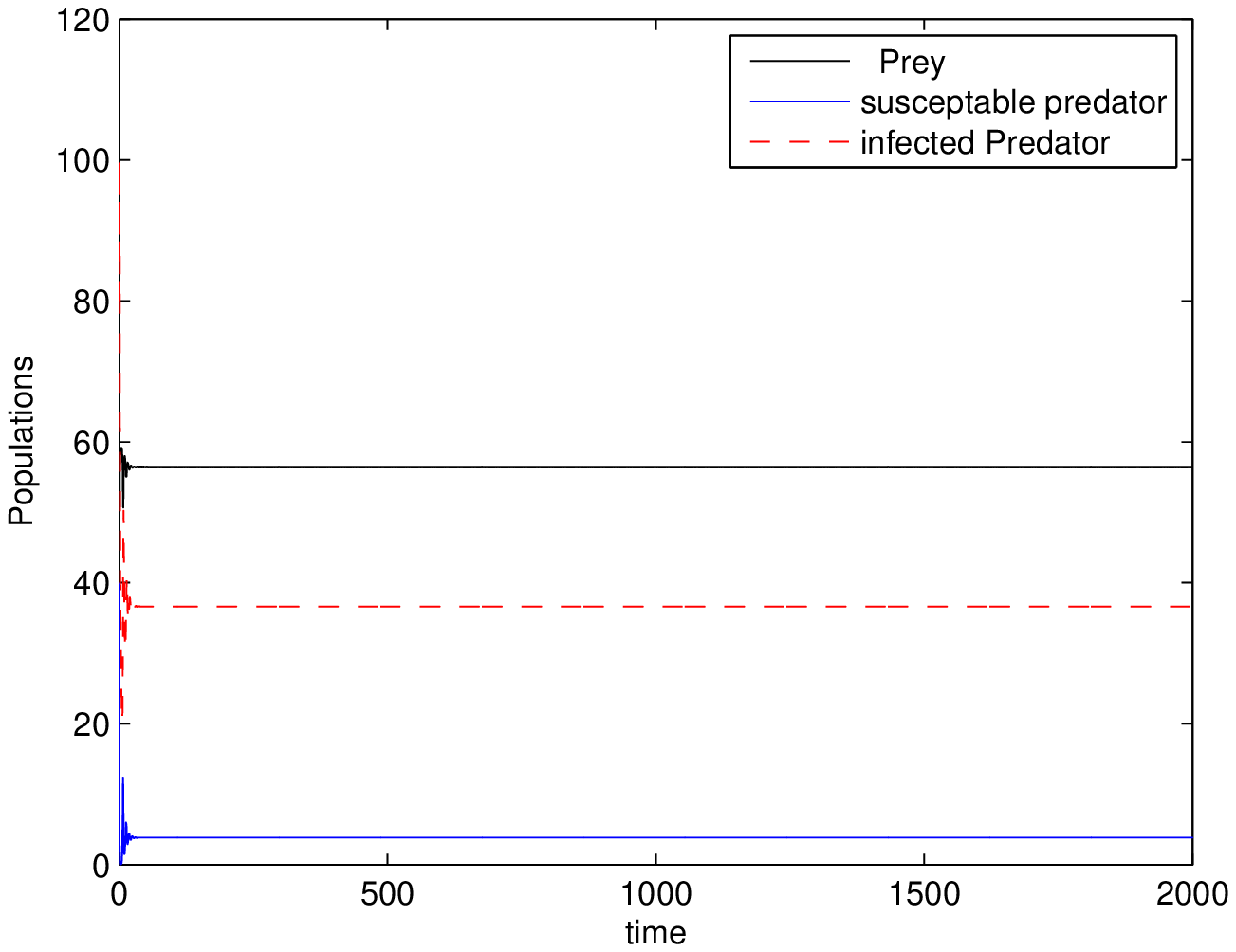}
(b)\includegraphics[width=5.5cm,height=3.5cm]{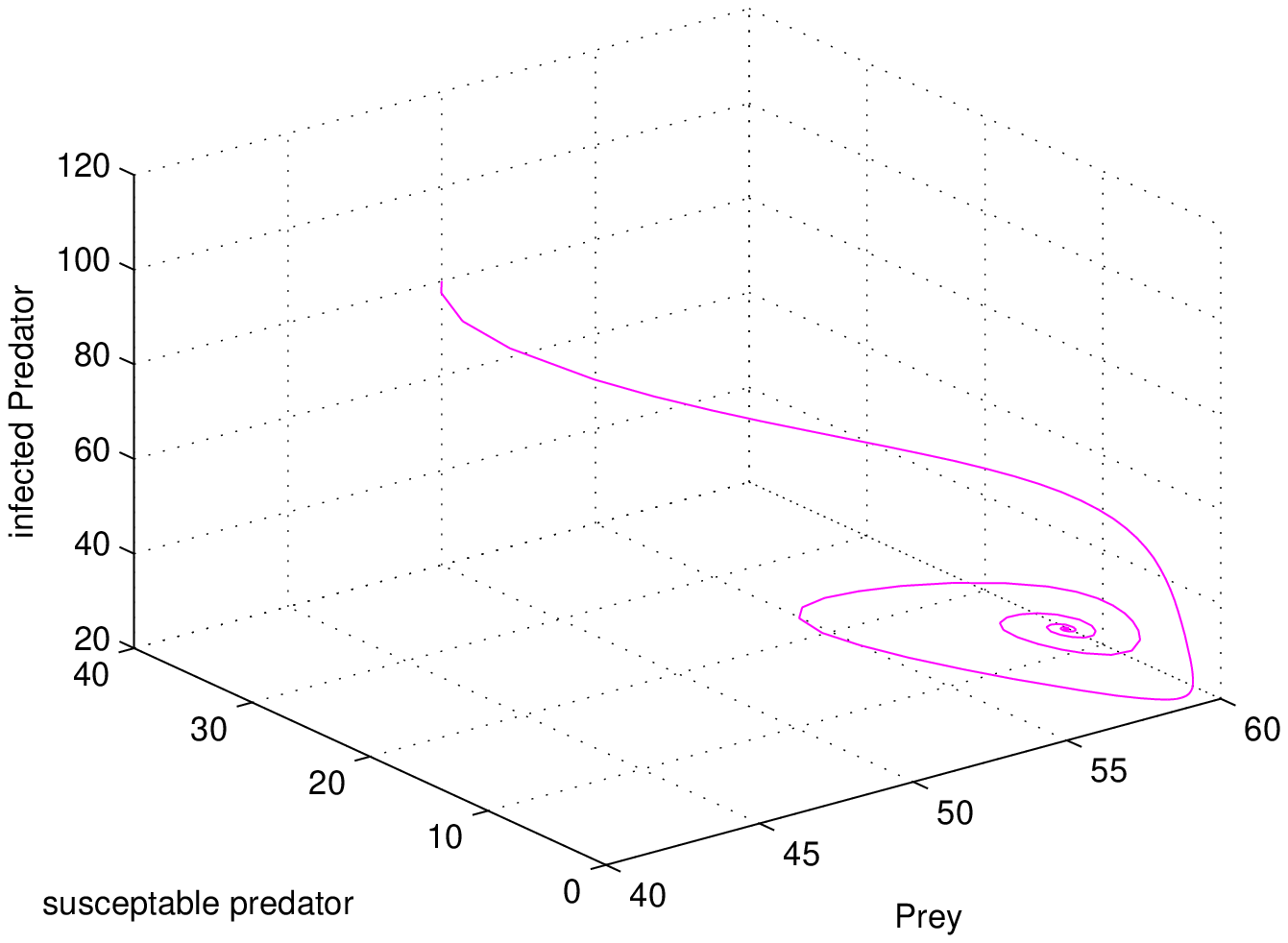}
\caption{\small{Stability behaviour of model the system (\ref{model-1}) around the equilibrium position $E_*$ with the initial conditions $ x_0=50, \ y_0=40, \ z_0=80$  and the set of parameter values  $S_1$, (a) Phase portrait, (b) Time series. Here $A_1=4.3463$, $A_2=2.7135$, $A_3 =4.8600$, $A_1A_2-A_3= 6.9339$.}} 
\label{fig50}
\end{figure}


\begin{figure}[htbp]
\centering
(a)\includegraphics[width=5.5cm,height=3.5cm]{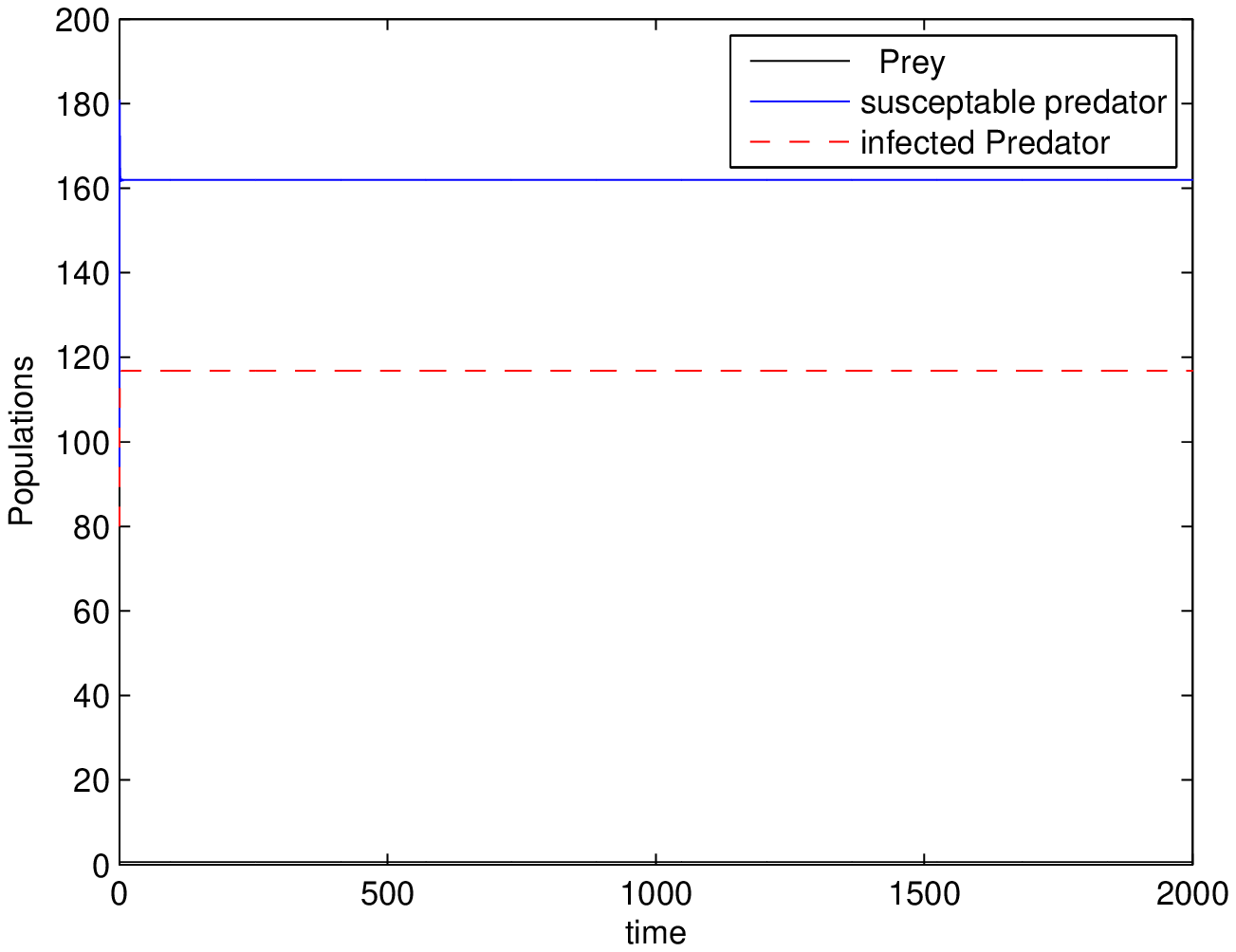}
(b)\includegraphics[width=5.5cm,height=3.5cm]{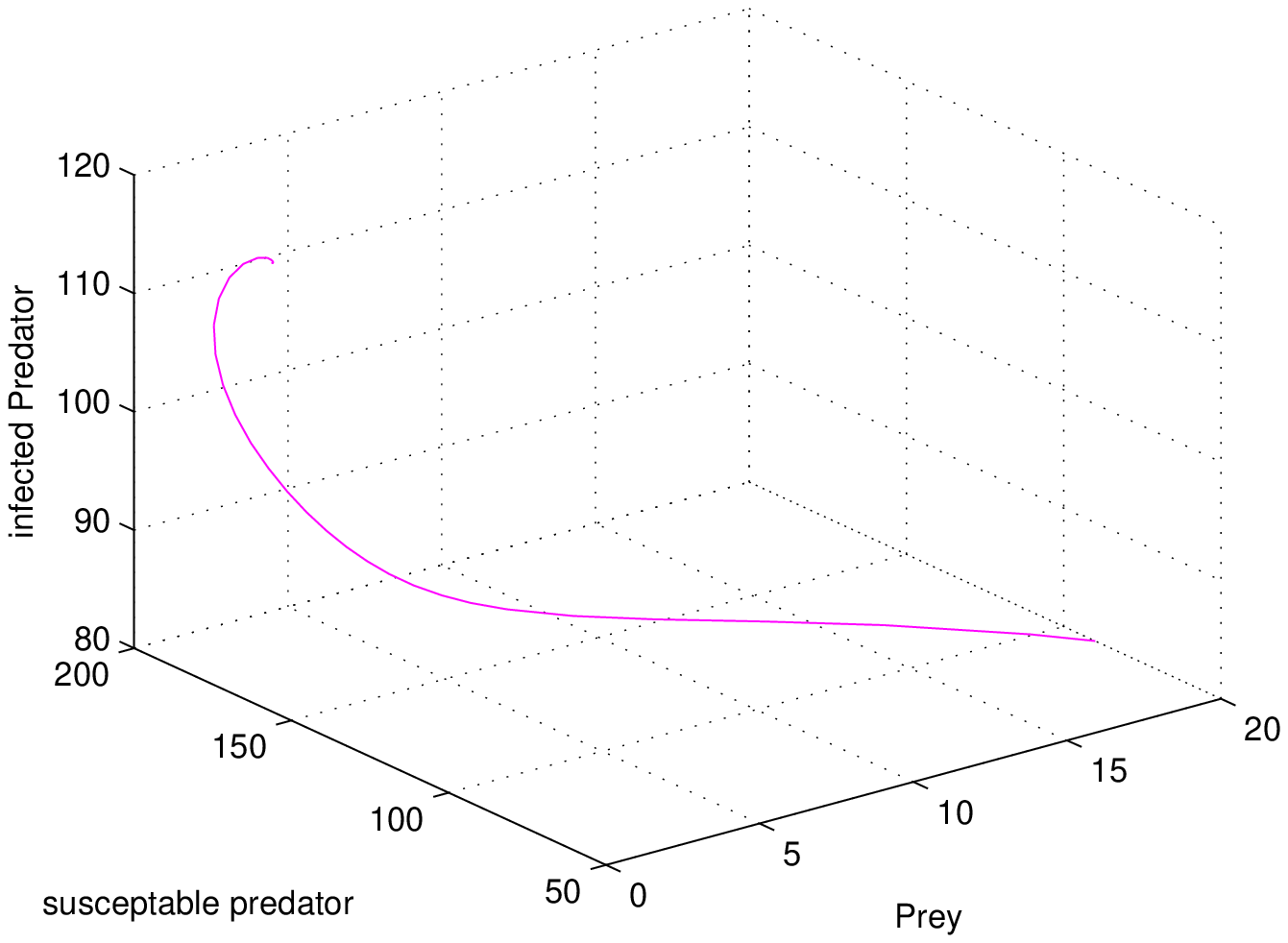}
\caption{\small{Stability behaviour of the system around the equilibrium position $E_*$ with the initial conditions $ x_0=20, \ y_0=90, \ z_0=80$ and the set of parameter values $S_3$, (a) Phase portrait, (b) Time series.}}\label{fig55} 
\end{figure}


\begin{figure}[htbp]
\centering
(a)\includegraphics[width=5.5cm,height=3.5cm]{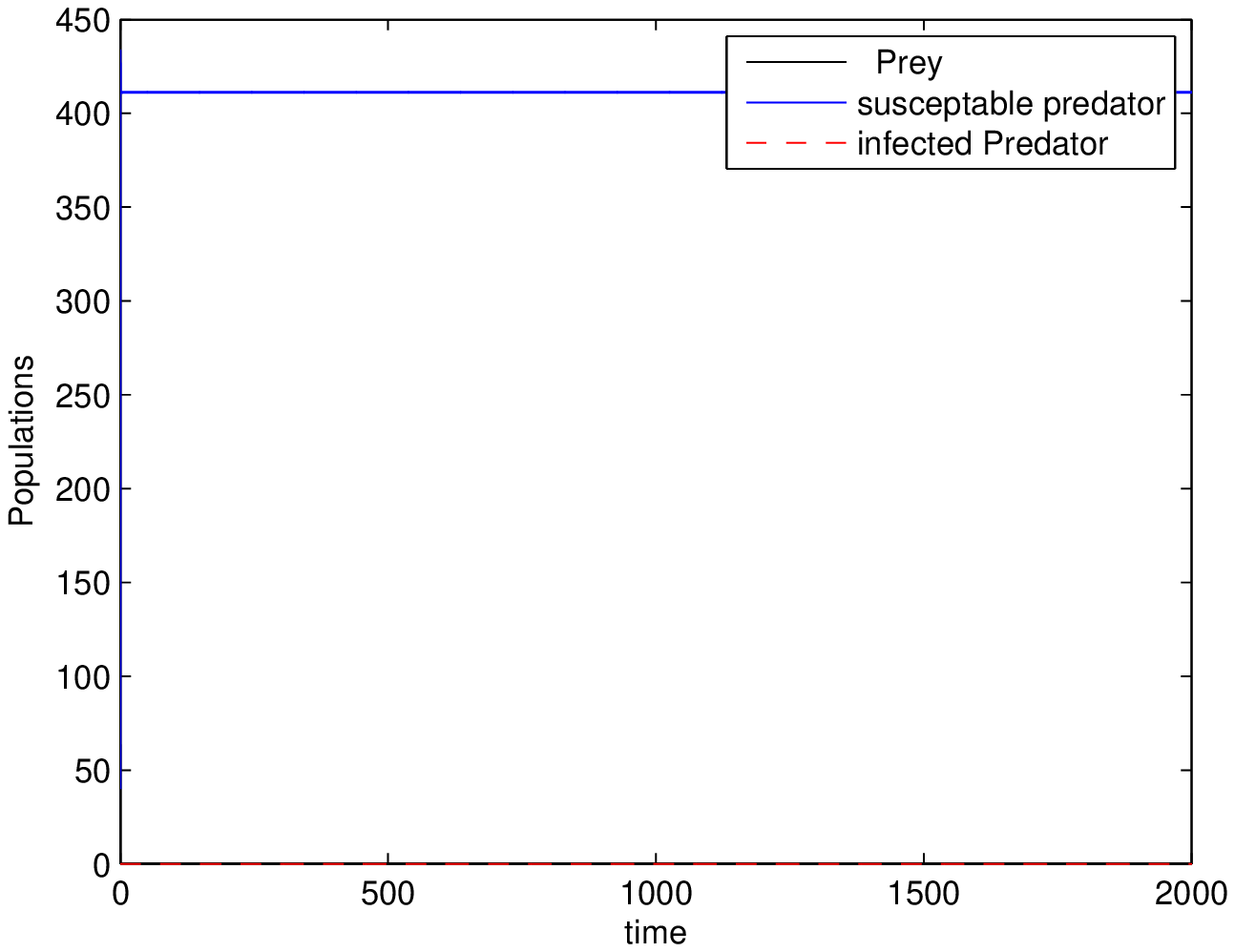}
(b)\includegraphics[width=5.5cm,height=3.5cm]{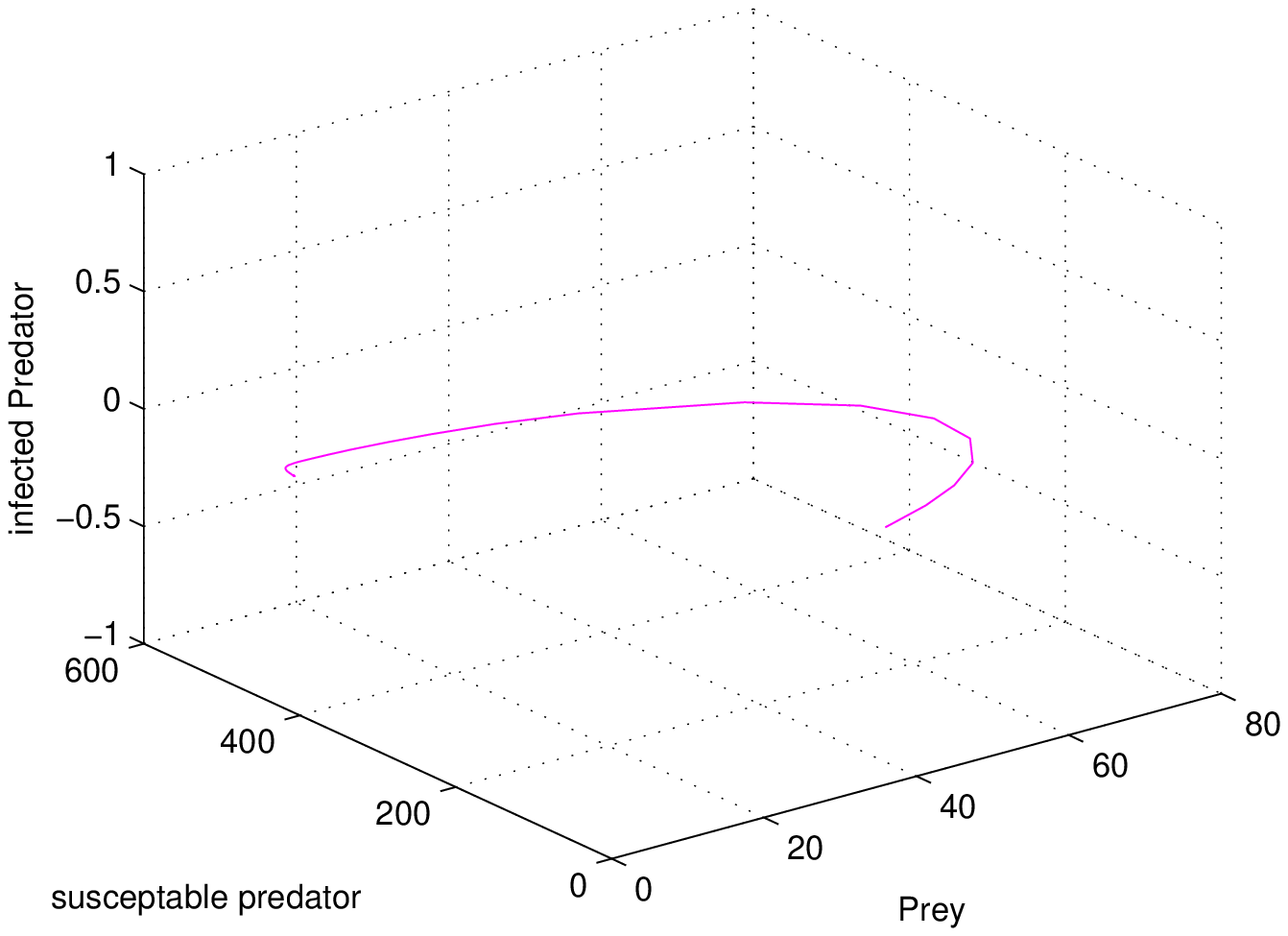}
\caption{\small{Stability behaviour of model the system around the equilibrium position $E_5$ with the initial conditions $ x_0=40, \ y_0=40, \ z_0=0$ and the set of parameter values $S_3$, (a) Phase portrait, (b) Time series. Here (i) $b_1x_5+\frac{c_1x_5y_5}{x_5^2+k_1}=5.0000>\frac{2c_1x_5^3y_5}{(x_5^2+k_1)^2}=0.02563$, 
(ii)$\frac{c_3y_5}{x_5+k_2}=12.6285 >\theta y_5+a_3=10.4237$.}} \label{fig53} 
\end{figure}

\begin{figure}[htbp]
\centering
(a)\includegraphics[width=5.5cm,height=3.5cm]{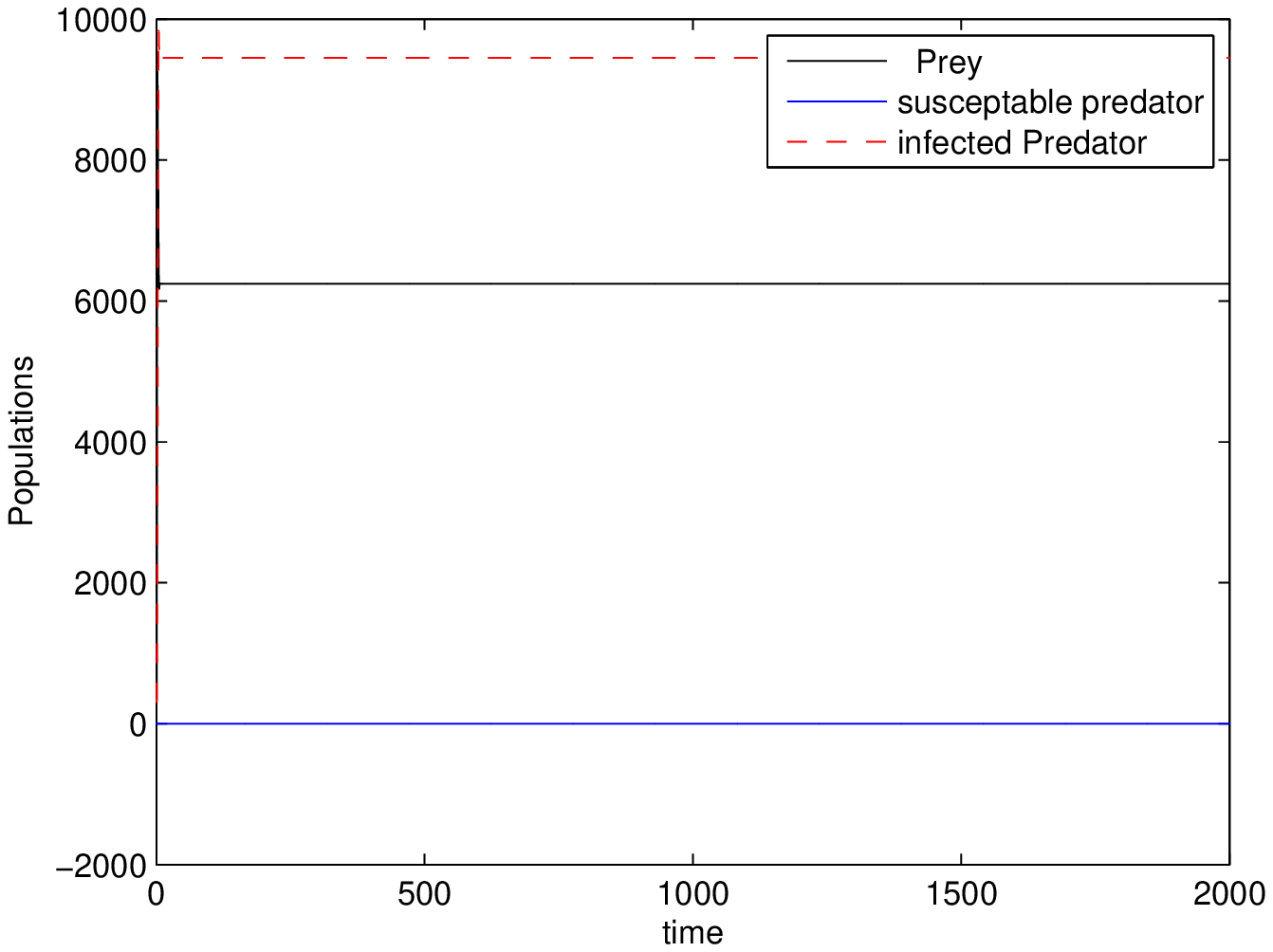}
(b)\includegraphics[width=5.5cm,height=3.5cm]{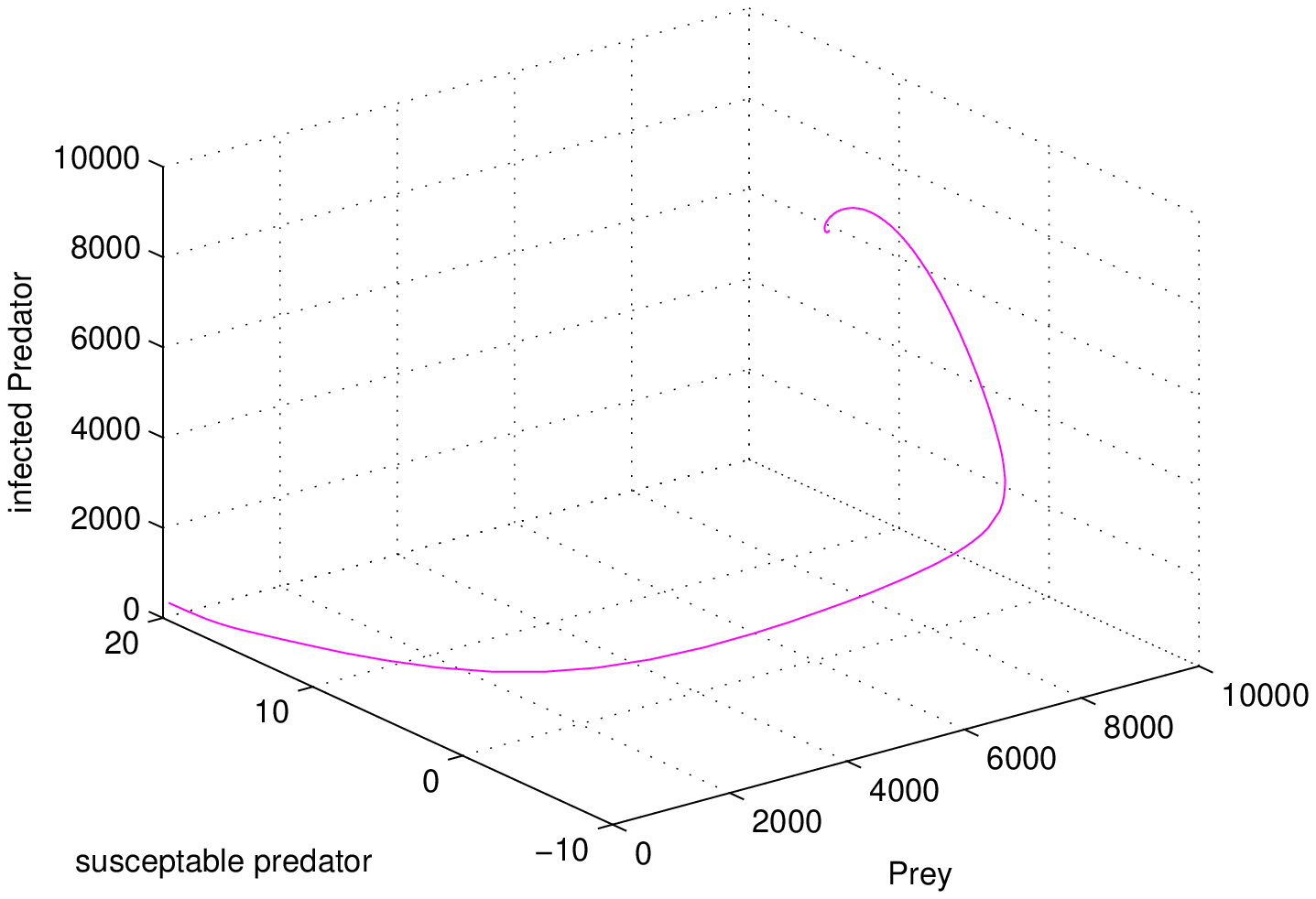}
\caption{\small{Stability behaviour of model the system around the equilibrium position $E_6$ with the initial conditions $ x_0=100, \ y_0=20, \ z_0=300$ and the set of parameter values $S_3$, (a) Phase portrait, (b) Time series. Here (i) $b_1x_6+\frac{pc_1x_6z_6}{x_6^2+k_1}= 5.0000 >\frac{2pc_1x_6^3z_6}{(x_6^2+k_1)^2}=3.7557$ , 
(ii)$\frac{c_2z_6}{x_6+k_2}+\theta z_6=206.614> a_2=7.8$.}} \label{fig54} 
\end{figure}\label{fig10}
\begin{figure}[htbp]
\centering
(a)\includegraphics[width=5.5cm,height=3.5cm]{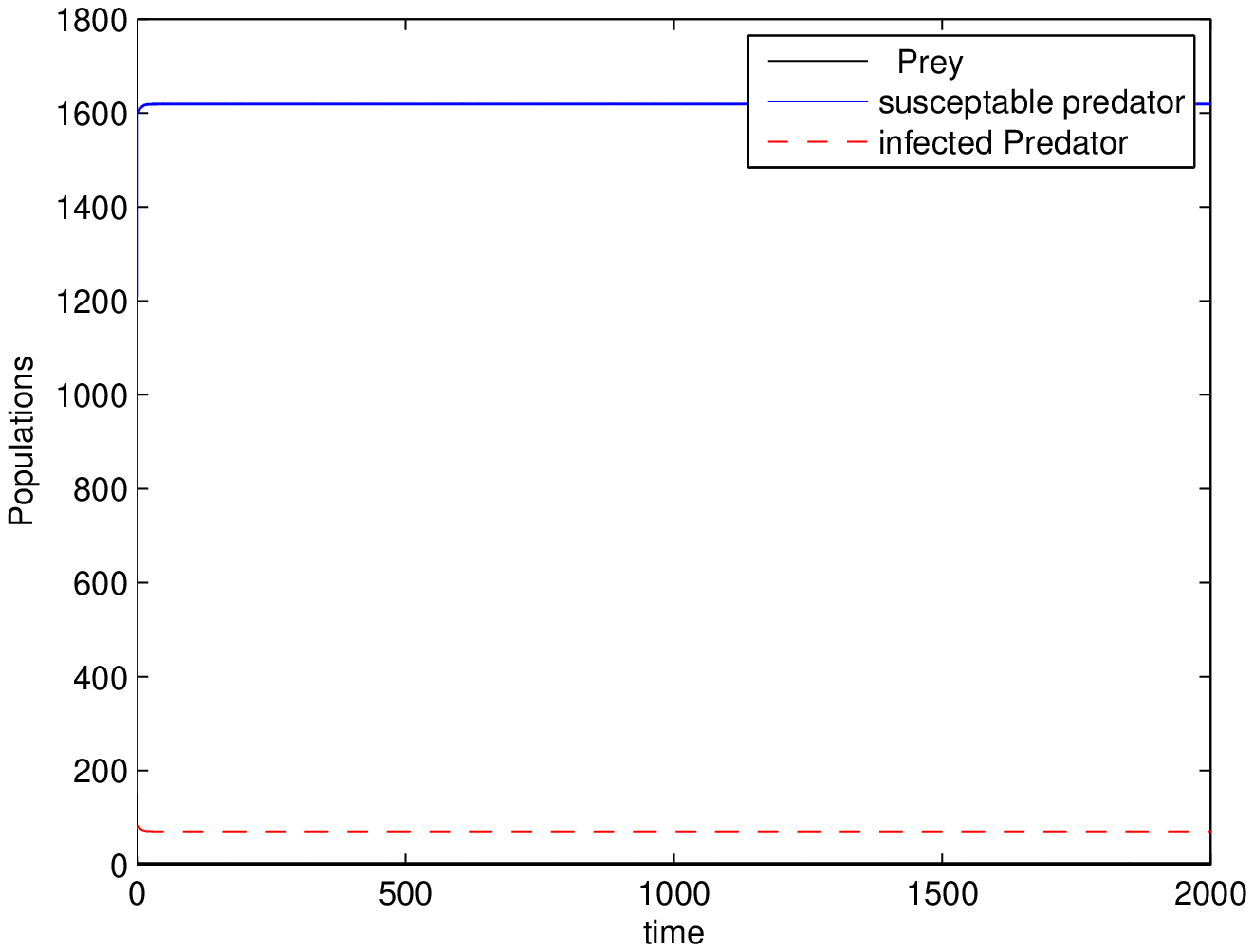}
(b)\includegraphics[width=5.5cm,height=3.5cm]{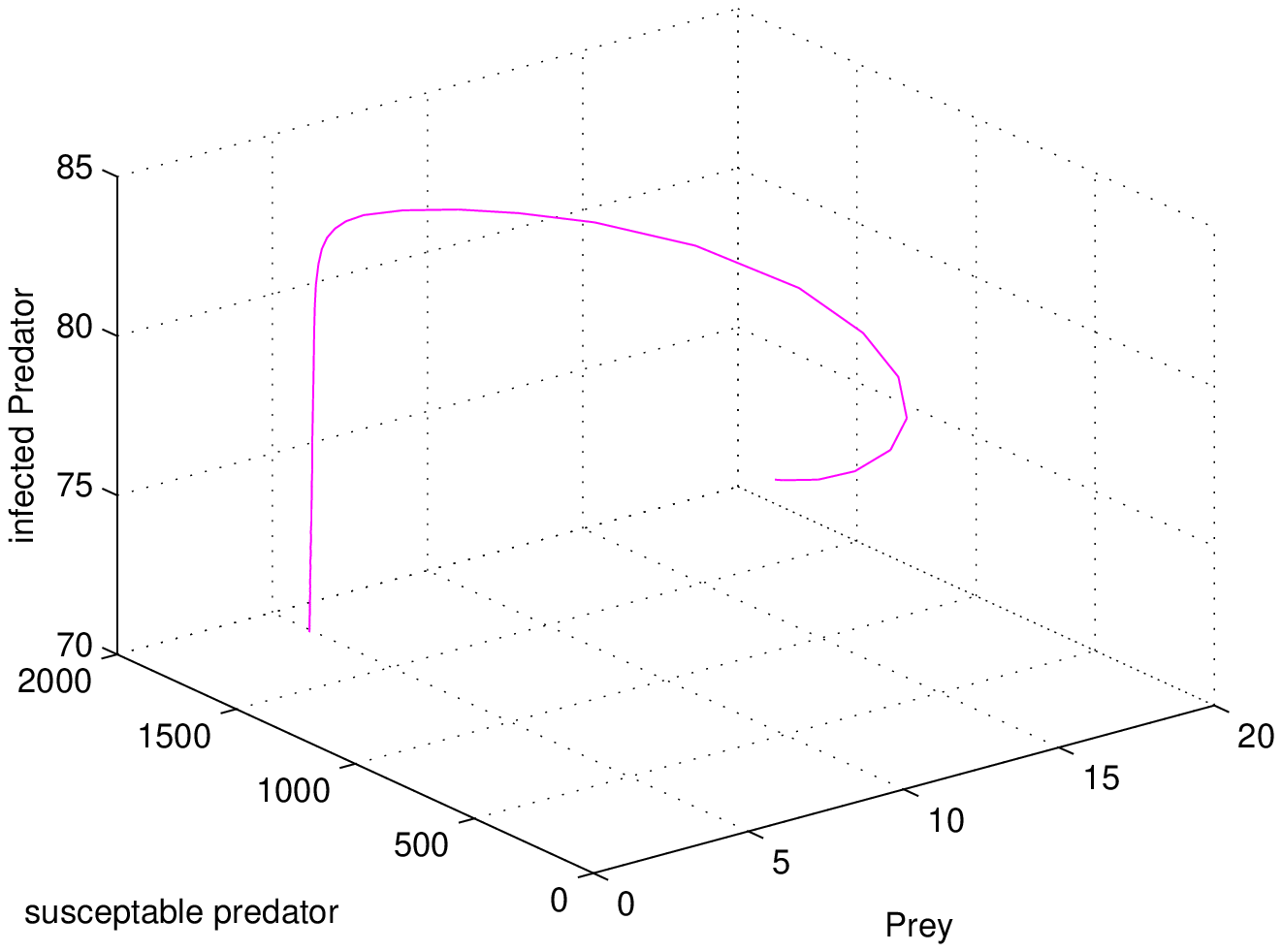}
\caption{\small{Stability behaviour of model the system around the equilibrium position $E_*^{I}$ with the initial conditions $ x_0=7, \ y_0=150, \ z_0=80$ and the set of parameter values $S_4$,  (a) Phase portrait, (b) Time series.}} \label{fig56} 
\end{figure}

\begin{figure}[htbp]
\centering
(a)\includegraphics[width=5.5cm,height=3.5cm]{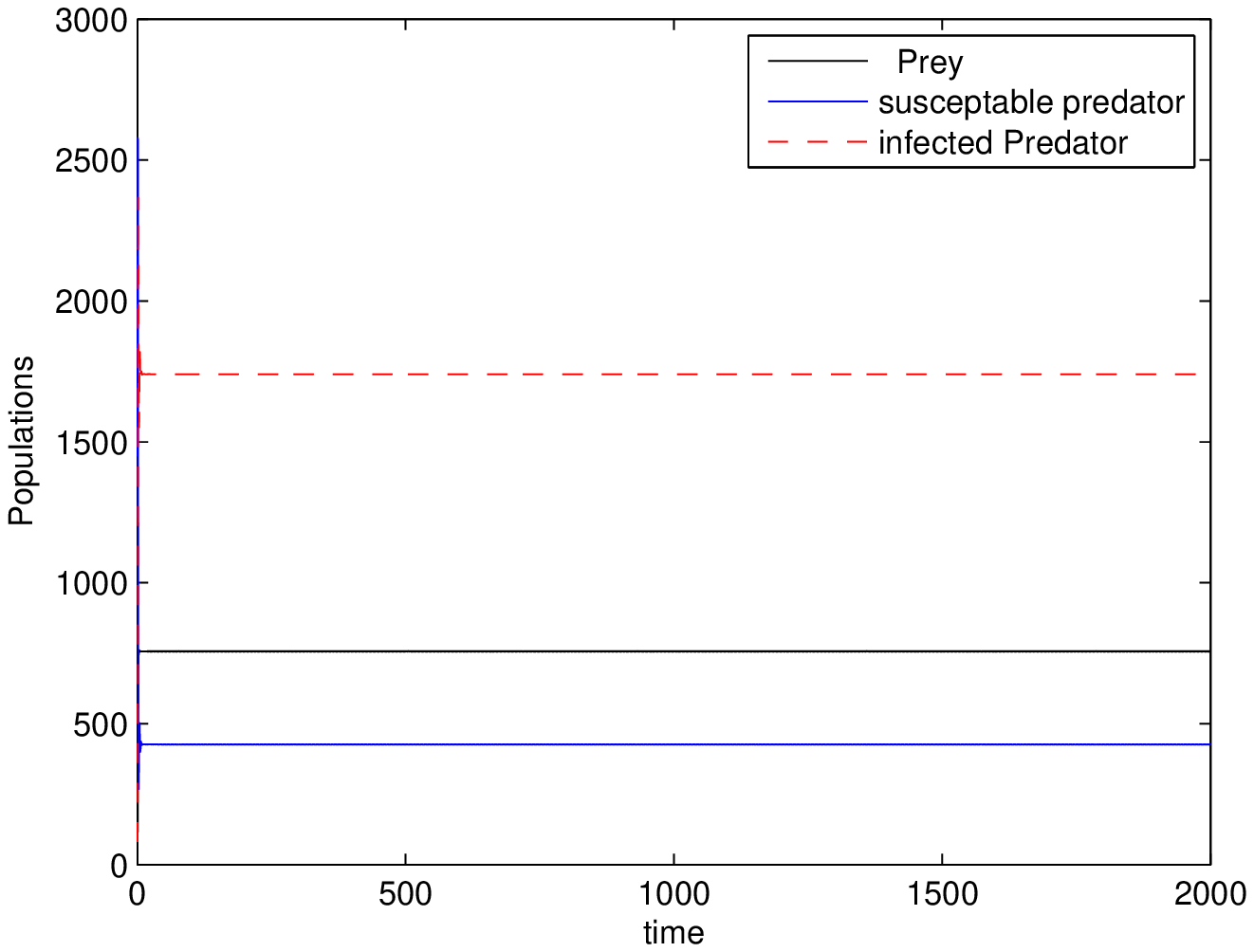}
(b)\includegraphics[width=5.5cm,height=3.5cm]{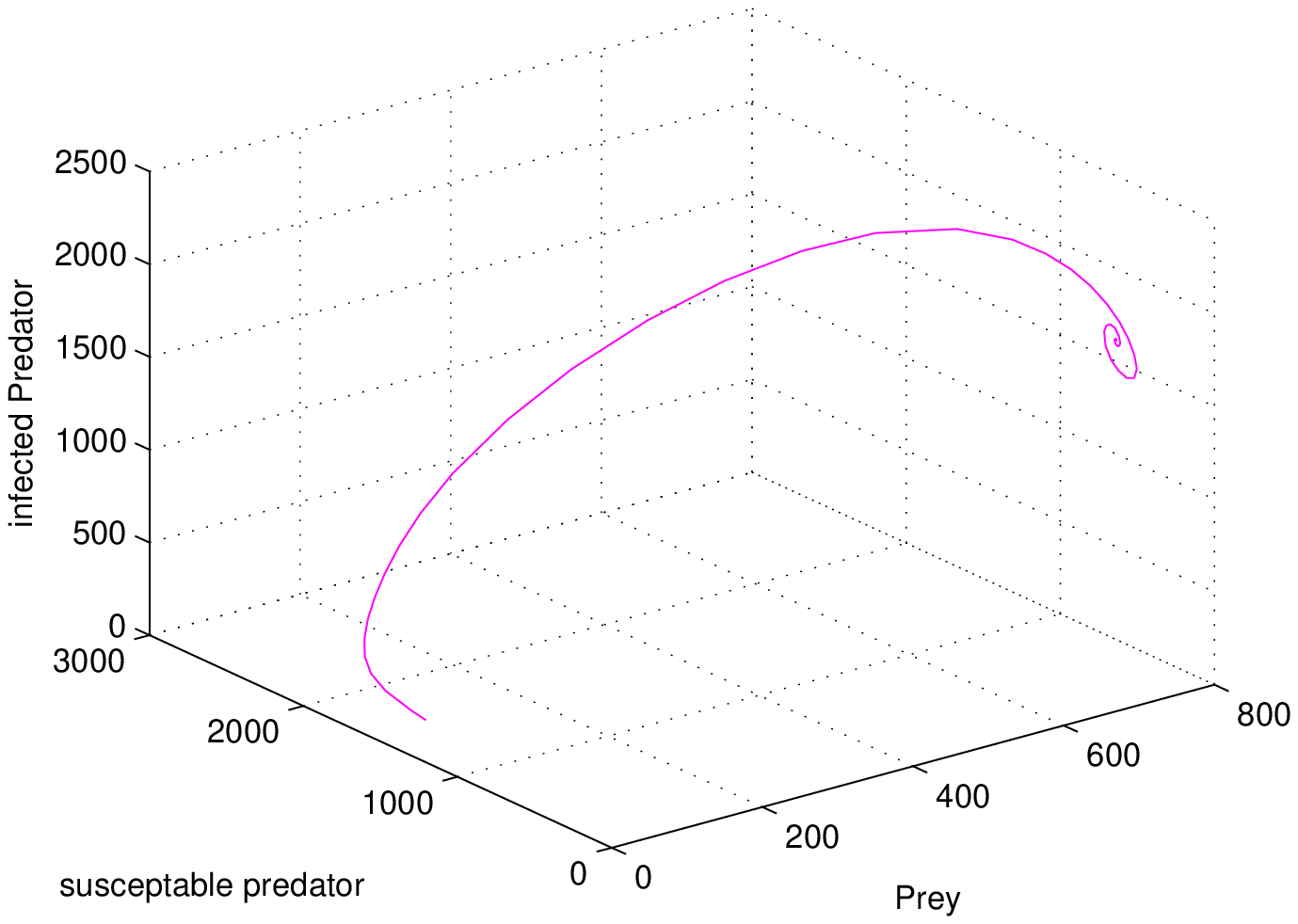}
\caption{\small{Stability behaviour of model the system around the equilibrium position $E_*^{II}$ with the initial conditions $ x_0=50, \ y_0=1450, \ z_0=80$ and the set of parameter values $S_4$,  (a) Phase portrait, (b) Time series.}} \label{fig57} 
\end{figure}

\begin{figure}[H]
\centering
(a)\includegraphics[width=5.5cm,height=3.5cm]{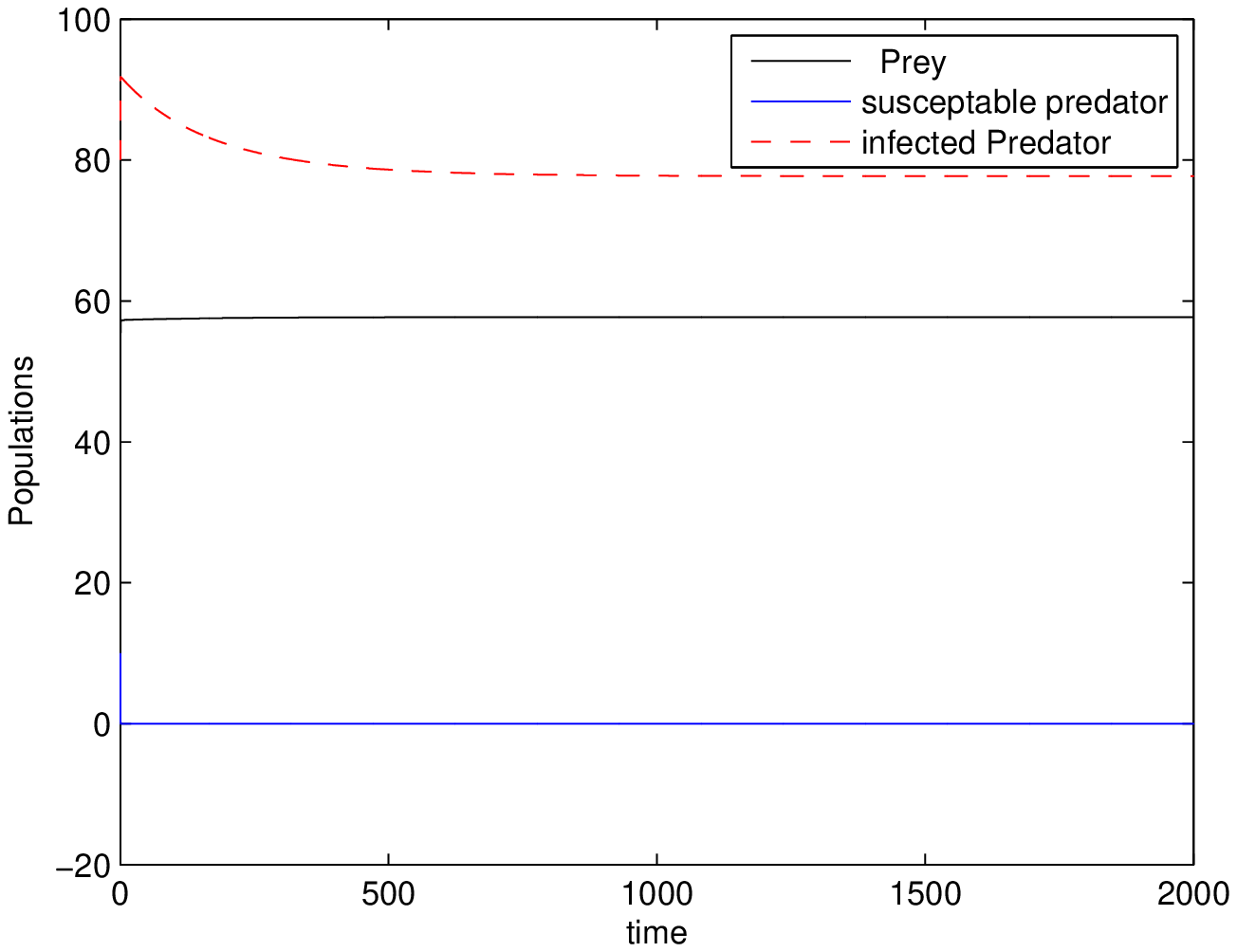}
(b)\includegraphics[width=5.5cm,height=3.5cm]{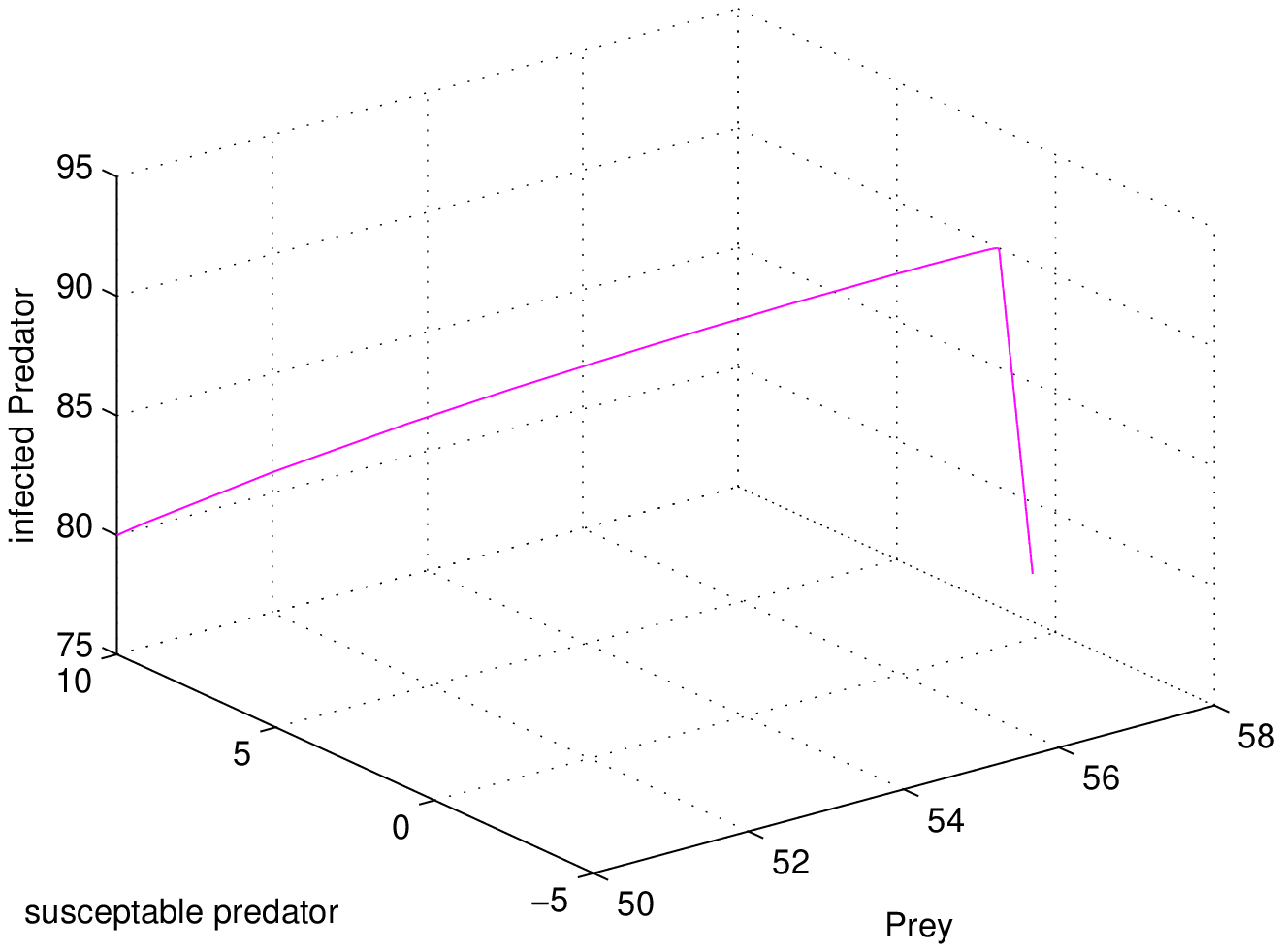}
\caption{\small{Stability behaviour of model the system around the equilibrium position $E_6$ with the initial conditions $ x_0=50, \ y_0=10, \ z_0=80$ and the set of parameter values $S_2$ , (a) Phase portrait, (b) Time series.}} 
\label{fig52}
\end{figure}


\begin{figure}[H]
\centering
(a)\includegraphics[width=5.5cm,height=3.5cm]{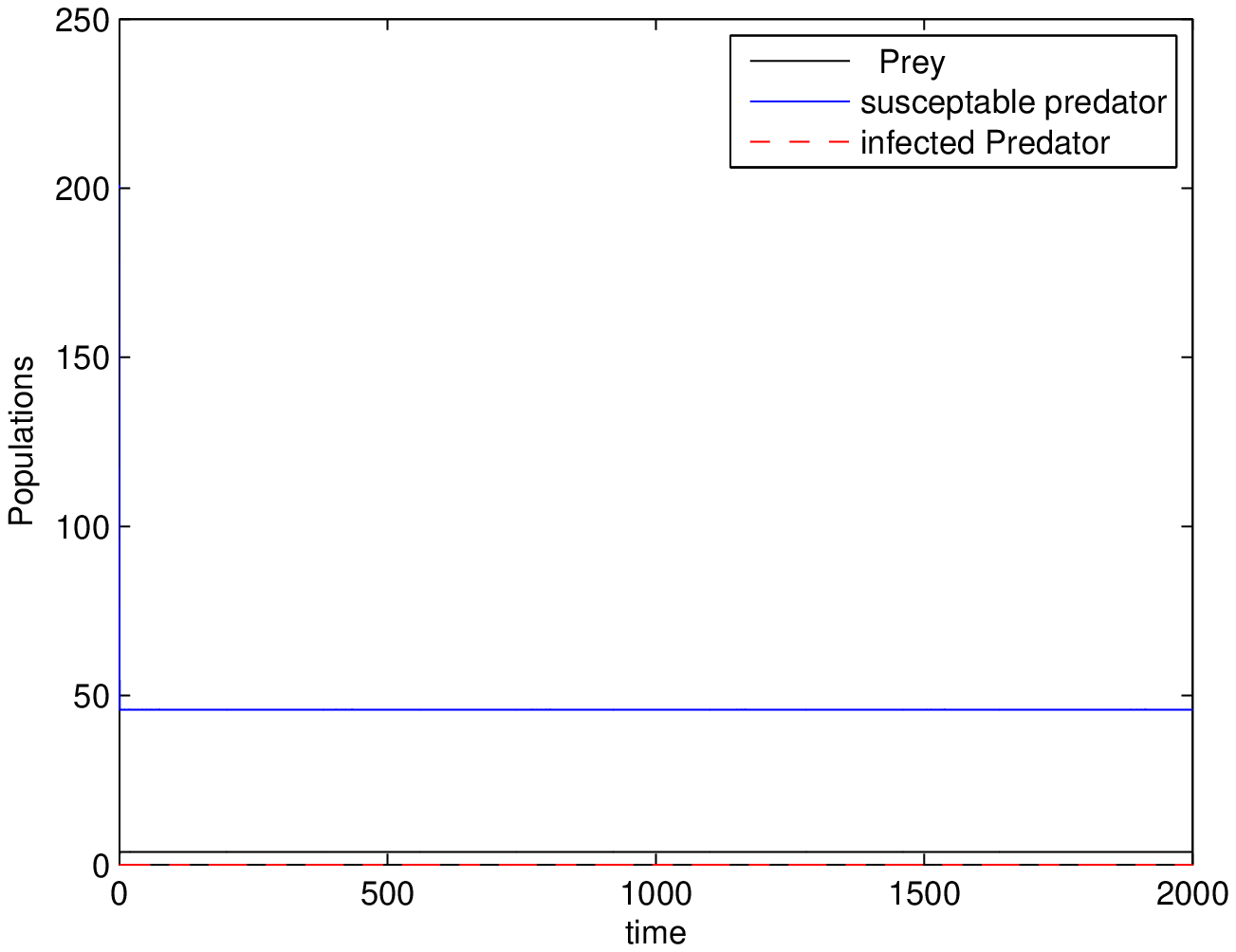}
(b)\includegraphics[width=5.5cm,height=3.5cm]{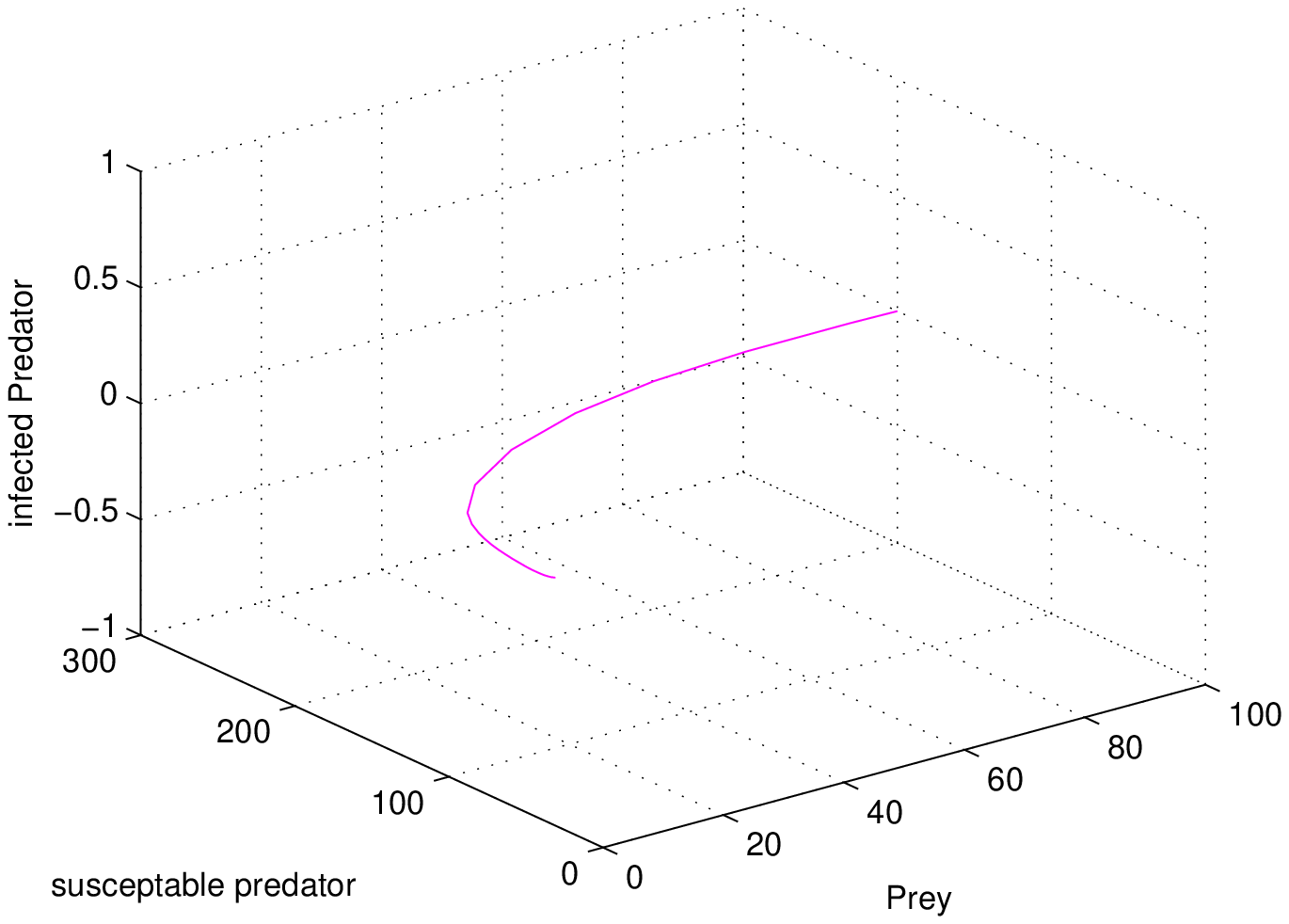}
\caption{\small{Stability behaviour of model the system around the equilibrium position $E_5$ with the initial conditions $ x_0=100, \ y_0=200, \ z_0=0 $ and the set of parameter values $S_2$,  (a) Phase portrait, (b) Time series.}} 
\label{fig51}
\end{figure}



\begin{table}[H]
\centering
\caption{Schematic representation of our analytical results for set of parameter values $S_1$}
\scalebox{0.8}{
\begin{tabular}{| l | l | l | l | l |}
\hline
\parbox[t]{0.5in}{Eq. pts.}
& \parbox[t]{1.5in}{Existence} & \parbox[t]{0.6in}{Feasibility} & \parbox[t]{0.6in}{Stability} & \parbox[t]{0.5in}{Figure}  \\ \hline 
$E_0$ & $(0, 0, 0)$ & Feasible & Unstable & ---------  \\ 
$E_1$ & $( 0, 0, 0.4102564103)$ & Feasible & Unstable & -------- \\ 
$E_2$ & $( 0, 308.6294416, 0)$ & Feasible & Unstable & ---------- \\ 
$E_3$ & $(0, 5.207637519, 35.24000437)$ & Feasible & Unstable & -------\\ 
$E_4$ & $(60, 0, 0)$ & Feasible & Unstable & --------- \\ 

$E_5^I$ & $(0.5159678886, 309.6247096, 0)$ & Feasible & --------- & --------- \\ 
$E_5^{II}$ & Does not exist & ----------- & --------- & --------- \\ 
$E_5^{III}$ & Does not exist & ---------- & ---------- & ---------\\ 
$E_6^I$ &\parbox[t]{1.5in}{$(59.98394610, 0,    0.5640614003)$} & Feasible &  Unstable & -- \\
$E_6^{II}$ & Does not exist & ------ & ------ & ----- \\ 
$E_6^{III}$ & Does not exist &-------- & ------- & -----\\ 
$E_*^I$ &\parbox[t]{1.5in}{$(56.43479200, 3.837197569,  \\36.62450011)$} & Feasible &\parbox[t]{1.5in}{Stability conditions are mentioned in section \ref{E_*}} & Figure \ref{fig50}\\ 
$E_*^{II}$ & Does not exist &-------- & -------- & ------- \\ 
$E_*^{III}$ & Does not exist & -------- & ------- &-------- \\ 
$E_*^{IV}$ & Does not exist  & -------- & ------ & -------- \\  \hline
 \end{tabular}}

\end{table}



\begin{table}[H]
\centering
\caption{Schematic representation of our analytical results for set of parameter values $S_2$}
\scalebox{0.8}{
\begin{tabular}{| l | l | l | l | l |}
\hline
\parbox[t]{0.5in}{Eq. pts.}
& \parbox[t]{1.5in}{Existence} & \parbox[t]{0.6in}{Feasibility} & \parbox[t]{0.6in}{Stability} & \parbox[t]{0.5in}{Figure}  \\ \hline
$E_0$ & $(0, 0, 0)$ & Feasible & Unstable & ---  \\
$E_1$ & $( 0, 0, 20)$ & Feasible & Unstable & ---\\ 
$E_2$ & $( 0, 38.57868020, 0)$ & Feasible & Unstable & - \\ 
$E_3$ & Does not exist & -------- & ------- & --------\\
$E_4$ & $(60, 0, 0)$ & Feasible & Unstable & ------- \\ 
$E_5^I$ & \parbox[t]{1.7in}{$(3.751381115,  45.81484682, 0)$}& Feasible & \parbox[t]{1.5in}{Stability conditions are mentioned in section \ref{E_5}} & Figure \ref{fig51}\\
$E_5^{II}$ & Does not exist & ---- & ---- & ----- \\ 
$E_5^{III}$ & Does not exist & ----- & -------- & ----\\ 
$E_6^I$ &\parbox[t]{1.7in}{$(57.70608071, 0,  77.70608071)$} & Feasible & \parbox[t]{1.5in}{Stability conditions are mentioned in section \ref{E_6}} & Figure \ref{fig52}\\
$E_6^{II}$ & Does not exist & ------ & ------ &------ \\ 
$E_6^{III}$ & Does not exist & -------- & ------ & -----\\ 
$E_*^I$ & Does not exist & -------- & -------- & ---------\\ 
$E_*^{II}$ & Does not exist & -------- & -------- & ---------\\ 
$E_*^{III}$ & Does not exist & ------ & ----- & ------\\  
$E_*^{IV}$ & Does not exist & ------ & ---- & ------ \\   \hline
\end{tabular}}
\end{table}





\begin{table}[H]
\centering
\caption{Schematic representation of our analytical results for set of parameter values $S_3$}
\scalebox{0.8}{
\begin{tabular}{| l | l | l | l | l |}
\hline
\parbox[t]{0.5in}{Eq. pts.}
& \parbox[t]{1.5in}{Existence} & \parbox[t]{0.6in}{Feasibility} & \parbox[t]{0.6in}{Stability} & \parbox[t]{0.5in}{Figure}\\
\hline
$E_0$ & $(0, 0, 0)$ & Feasible & Unstable & -----  \\ 
$E_1$ & $( 0, 0, 82.50000000)$ & Feasible & Unstable & ----- \\ 
$E_2$ & $( 0, 408.5714286, 0)$ & Feasible & Unstable & ---- \\
$E_3$ &$(0, 164.3476956, 114.3012791)$ & Feasible & Unstable & -----\\ 
$E_4$ & $(10000, 0, 0)$ & Feasible & Unstable & ---- \\ 
$E_5^I$ &\parbox[t]{1.7in} {$(0.3585095936, 411.2346427, 0)$} & Feasible & \parbox[t]{1.5in}{Stability conditions are mentioned in section \ref{E_5}}& Figure \ref{fig53} \\ 
$E_5^{II}$ & Does not exist & ------ & ------ & ----- \\ 
$E_5^{III}$ & Does not exist & ----- & ----- & ---- \\  
$E_6^I$ &\parbox[t]{1.7in}{$(6244.212048, 0,  9448.818072)$} & Feasible & \parbox[t]{1.5in}{Stability conditions are mentioned in section \ref{E_6}} & Figure \ref{fig54}\\
$E_6^{II}$ &\parbox[t]{1.7in}{$(2.657590193, 0,  86.48638529)$} & Feasible & Unstable  & ---- \\ 
$E_6^{III}$ &\parbox[t]{1.7in}{$(30.13036196, 0,  127.6955429)$} & Feasible & Unstable  & ------\\ 
$E_*^I$ &\parbox[t]{1.7in}{$(0.5990751338, 161.9321592, \\116.8375923)$} & Feasible & \parbox[t]{1.5in}{Stability conditions are mentioned in section \ref{E_*}}& Figure \ref {fig55}\\ 
$E_*^{II}$ &\parbox[t]{1.7in}{$(73.69033011,  33.00904773, \\252.2068593)$} & Feasible & Unstable & ----- \\ 
$E_*^{III}$ & Does not exist & ------- & ------ & ----\\ 
$E_*^{IV}$ & Does not exist & ----- & ----- & ----\\ \hline
\end{tabular}}
\end{table}




\begin{table}[H]
\centering
\caption{Schematic representation of our analytical results for set of parameter values $S_4$}
\scalebox{0.8}{
\begin{tabular}{| l | l | l | l | l |}
\hline
\parbox[t]{0.5in}{Eq. pts.}
& \parbox[t]{1.5in}{Existence} & \parbox[t]{0.6in}{Feasibility} & \parbox[t]{0.6in}{Stability} & \parbox[t]{0.5in}{Figure}  \\ \hline
$E_0$ & $(0, 0, 0)$ & Feasible & Unstable & -----  \\ 
$E_1$ & $( 0, 0, 20)$ & Feasible & Unstable & -------- \\ 
$E_2$ & $( 0, 1714.285714, 0)$ & Feasible & Unstable & ------- \\ 
$E_3$ &$(0, 1637.974544, 44.24202326)$ & Feasible & Unstable& -----\\ 
$E_4$ & $(800, 0, 0)$& Feasible & Unstable & ----- \\ 
$E_5^I$ & $(3.142771393, 1741.223755, 0)$ & Feasible & Unstable  & -----\\ 
$E_5^{II}$ & $(41.15227929, 2067.019537, 0)$ & Feasible & Unstable & ---- \\ 
$E_5^{III}$ & $(618.5620922, 7016.246504, 0)$ & Feasible & Unstable & ------\\ 
$E_6^I$ & $(798.1394249, 0, 99.81394249)$ & Feasible & Unstable &-------\\
$E_6^{II}$ & Does not exist & ---- & ------ & ------- \\ 
$E_6^{III}$ & Does not exist & ------- & ----- & ---------\\ 
$E_*^I$ &\parbox[t]{1.7in}{$(3.271695068,1618.749669,\\71.15684924)$} & Feasible &\parbox[t]{1.5in}{Stability conditions are mentioned in section \ref{E_*}}& Figure \ref{fig56}\\ 
$E_*^{II}$ &\parbox[t]{1.7in}{$(756.7723630,426.6171136,\\1740.142426)$} & Feasible & \parbox[t]{1.5in}{Stability conditions are mentioned in section \ref{E_*}} & Figure \ref{fig57}\\ 
$E_*^{III}$ & \parbox[t]{1.7in}{$(33.12282349,1461.965361,\\290.6548798)$} & Feasible & Unstable & -------\\ 
$E_*^{IV}$ & Does not exist & ---- & ------ & ------- \\  \hline
\end{tabular}}
\end{table}


\section{Conclusions and comments}
In this paper, we have proposed and analyzed an  eco-epidemiological model where only the predator population is infected by an infectious disease. Here we have considered a modified Leslie-Gower and Holling type-III predator-prey model. We have divided the predator population into two sub classes: susceptible and infected. Then we study the behaviour of the system at various equilibrium points and their stability. The conditions for existence and stability of all the equilibria of the system have been given. The system (\ref{model-1}) has eight equilibrium points: one trivial equilibrium $E_0$, three axial equilibrium points $E_1$, $E_2$, $E_4$, three planar equilibrium points $E_3$, $E_5$, $E_6$ and one coexistence equilibrium $E_*$. For our model: $E_i$, $i=0,1,2,4$ exist and are unstable for all times. $E_3$ exists if $a_3k_2\theta < a_2c_3-a_3c_2< a_2k_2\theta$ and $c_2>c_3$ but unstable. The equilibrium point $E_5$ is locally asymptotically stable if $b_1x_5+\frac{c_1x_5y_5}{x_5^2+k_1}>\frac{2c_1x_5^3y_5}{(x_5^2+k_1)^2}$, $\frac{c_3y_5}{x_5+k_2}>\theta y_5+a_3$. Also  $E_6$ is locally asymptotically stable if $b_1x_6+\frac{pc_1x_6z_6}{x_6^2+k_1}>\frac{2pc_1x_6^3z_6}{(x_6^2+k_1)^2}$, $\frac{c_2z_6}{x_6+k_2}+\theta z_6>a_2$. The coexistence equilibrium point $E_*$ is locally as well as globally asymptotically stable under some conditions. Persistence of the system is also shown.\\
\indent At last, we conclude that our eco-epidemic predator--prey model with infected predator exhibits very interesting dynamics. Here we have assumed Holling type III response mechanism for predation.  So, we can refine the model considering other type of functional response. We can also consider the disease infection in the prey population, which can give us a very rich dynamics. There must be some time lag, called gestation delay. So, as part of future work to improve the model we can incorporate the gestation delay in our model to make it more realistic.


\end{document}